\documentclass[12pt,reqno]{amsart}
\usepackage{amscd,amsmath,amsopn,amssymb,amsthm,multicol}
\usepackage{hyperref}

\DeclareMathOperator{\Ad}{Ad}
\DeclareMathOperator{\ad}{ad}

\newtheorem{theorem}{Theorem}
\newtheorem{lemma}{Lemma}
\newtheorem{remark}{Remark}
\newtheorem{proposition}{Proposition}
\newtheorem{definition}{Definition}

\newtheorem{corollary}{Corollary}

\begin{document}

\title[ Riemannian $M$-spaces with homogeneous geodesics ]
{ Riemannian $M$-spaces with homogeneous geodesics }
\author{Andreas Arvanitoyeorgos, Yu Wang$^*$ and Guosong Zhao}

\address{University of Patras, Department of Mathematics, GR-26500 Patras, Greece}
\email{arvanito@math.upatras.gr}
\address{Sichuan university of Science and Engineering, Zigong, 643000, China}

\email{wangyu\_813@163.com \\
$^*$ Corresponding author.}

\address{Sichuan university, Chengdu, 610064, China}
\email{gszhao@scu.edu.cn}

\begin{abstract}
We investigate homogeneous geodesics in a class of homogeneous spaces called $M$-spaces, which are defined as follows.  Let $G/K$ be a generalized flag manifold with
$K=C(S)=S\times K_1$,
where $S$ is a torus in a compact simple Lie group $G$ and $K_1$ is the semisimple part of $K$.
Then the {\it associated $M$-space} is the homogeneous space $G/K_1$.  These spaces were introduced and studied by H.C. Wang in 1954.
We prove that for various classes of $M$-spaces the only g.o. metric is the standard metric.
For other classes of $M$-spaces we give either necessary, or necessary and sufficient conditions, so that a $G$-invariant metric on $G/K_1$ is a g.o. metric.
The analysis is based on properties of the isotropy representation
$\mathfrak{m}=\mathfrak{m}_1\oplus \cdots\oplus \mathfrak{m}_s$  of the flag manifold $G/K$ (as $\Ad(K)$-modules) and corresponding decomposition $\mathfrak{n}=\mathfrak{s}\oplus\mathfrak{m}_1\oplus \cdots\oplus \mathfrak{m}_s$ of the tangent space of the $M$-space $G/K_1$ (as $\Ad(K_1)$-modules).

 \medskip


\medskip

 \end{abstract}
\maketitle

\section*{Introduction}
 Let $(M, g)$  be a homogeneous Riemannian manifold, i.e. a connected Riemannian manifold on which the largest connected group $G$ of isometries acts transitively. Then $M$ can be expressed as a homogeneous space $(G/K, g)$, where $K$ is the isotropy group at a fixed pointed $o$ of $M$, and $g$ is a $G$-invariant metric.
 A geodesic $\gamma(t)$ through the origin $o$ of $M=G/K$ is called {\it homogeneous} if it is an orbit of a one-parameter subgroup of $G$, that is
\begin{equation}\label{1}
 \gamma(t)=\mathrm{exp}(tX)(o), \quad t \in \mathbb{R},
 \end{equation}
where $X$ is a non zero vector of $\mathfrak{g}$.

A homogeneous Riemannian manifold $M=G/K$ is called a {\it g.o. space}, if all geodesics are homogeneous with respect to the largest connected group of isometries $I_o(M)$.  A $G$-invariant  metric $g$ on $M$ is called {\it $G$-g.o.}  if all geodesics are homogeneous   with respect to the  group $G \subseteq I_o(M)$.  Of course a $G$-g.o. metric is a g.o. metric, but the converse is not true in general. {\it In this paper we only consider $G$-g.o. metrics, which we also call them g.o. metrics}.

Naturally reductive spaces, symmetric spaces and weekly symmetric spaces   are g.o. spaces (\cite{Be-Ko-Va}, \cite{Car}, \cite{Ko-No}, \cite{Zil}).
In \cite{Ko-Pr-Va} O. Kowalski, F. Pr\" ufer and L. Vanhecke gave an explicit classification of all naturally reductive spaces up to dimension five. In \cite{Ko-Va} O. Kowalski and L. Vanhecke gave a classification of all g.o. spaces, which are in no way naturally reductive, up to dimension six.
In \cite{Gor} C. Gordon described g.o. spaces  which are nilmanifolds, and in \cite{Tam} H. Tamaru  classified homogeneous g.o. spaces which are fibered over irreducible symmetric spaces.
In  \cite{Du2} and \cite{Du-Ko1} O. Kowalski and Z. Du\v sek investigated  homogeneous geodesics in   Heisenberg groups  and some $H$-type groups. Examples of g.o. spaces in dimension seven were obtained by Du\v sek, O. Kowalski and S. Nik\v cevi\' c in \cite{Du-Ko-Ni}.
In \cite{Al-Ar} the first  author and D.V. Alekseevsky  classified generalized flag manifolds which are   g.o. spaces. Recently, the first two authors classified generalized  Wallach spaces which are g.o. spaces (\cite{Ar-Wa}).  Also, in \cite{Ch-Ni} Z. Chen and Yu. Nikonorov classified compact simply connected g.o. spaces  with two isotropy summands.
Other interesting results about g.o. spaces  can be found in  \cite{Al-Ni}, \cite{Arv}, \cite{Ca-Ma}, \cite{Du3},  \cite{Du-Ko-Vl},  \cite{Ko-Sz}, \cite{Lat}, \cite{Mar}, \cite{Sze3}, \cite{Sze4} and \cite{Wa-Zh}. Finally, the notion of homogeneous geodesics can be extended to geodesics which are orbits of a product of two exponential factors (cf. \cite{Ar-So1}, \cite{Ar-So2}).

 The general problem of classification of compact homogeneous Riemannian manifolds $(M=G/K, g)$ with homogeneous geodesics remains open.

The object of  the present paper is to study homogeneous geodesics in $M$-spaces.
These spaces were introduced and studied by H.C. Wang in \cite{Wan} and
are defined as follows:  Let $G/K$ be a generalized flag manifold with
$K=C(S)=S\times K_1$,
where $S$ is a torus in a compact simple Lie group $G$ and $K_1$ is the semisimple part of $K$.
Then the {\it corresponding $M$-space} is the homogeneous space $G/K_1$.

 Let $\mathfrak{g}$  and  $\mathfrak{k}$ be the Lie algebras of the Lie groups $G$ and $K$ respectively. Let $\mathfrak{g}=\mathfrak{k}\oplus \mathfrak{m}$ be an $\Ad(K)$-invariant reductive decomposition of the Lie algebra $\mathfrak{g}$, where $\mathfrak{m}\cong T_o(G/K)$.  This is orthogonal with respect to $B=-$Killing from on $\mathfrak{g}$.
Assume that
 \begin{equation}\label{2}
 \mathfrak{m}=\mathfrak{m}_1\oplus \cdots\oplus \mathfrak{m}_s
 \end{equation}
  is a $B$-orthogonal decomposition of $\mathfrak{m}$ into pairwise inequivalent irreducible $\mathrm{ad}(\mathfrak{k})$-modules.

  Let $G/K_1$ be the corresponding $M$-space  and
   $\mathfrak{s}$ and $ \mathfrak{k}_1$ be  the Lie algebras of $S$ and $K_1$ respectively. We denote by $\mathfrak{n}$  the tangent space $T_o(G/K_1)$, where $o=eK_1$, it follows that $\mathfrak{n}=\mathfrak{s}\oplus\mathfrak{m}$.  A $G$-invariant metric $g$  on $G/K_1$ induces a scalar product $\langle \cdot,\cdot \rangle$ on $\mathfrak{n}$ which is $\Ad({K}_1)$-invariant.  Such an  $\Ad({K}_1)$-invariant scalar product $\langle \cdot,\cdot \rangle$ on $\mathfrak{n}$ can be expressed as $\langle x, y \rangle=B(\Lambda x, y)\; (x, y \in \mathfrak{n})$, where $\Lambda$ is an $\Ad({K}_1)$-equivariant positive definite symmetric operator on $\mathfrak{n}$. Conversely, any such operator $\Lambda$ determines an  $\Ad({K}_1)$-invariant scalar product $\langle x, y \rangle=B(\Lambda x, y)$ on $\mathfrak{n}$, which in turn determines   a $G$-invariant Riemannian metric $g$ on $\mathfrak{n}$. We say that  $\Lambda$ is the {\it operator associated} to the metric $g$, or simply the {\it associated operator}.
A  Riemannian metric generated by the inner product $B(\cdot,\cdot)$ is called {\it standard metric}.  An $M$-space $G/K_1$ with the standard metric is g.o.,  since it is naturally reductive.

\smallskip
The main results of the paper are the following:

\begin{theorem} \label{T1} Let  $G/K$ be a generalized flag manifold  with $s\geq 3 $ in the decomposition  (\ref{2}). Let $G/K_1$ be the corresponding $M$-space.
 If  $(G/K_1, g)$  is a  g.o. space,  then
 $$g=\langle \cdot,\cdot\rangle=\Lambda\mid_{\mathfrak{s}}+\lambda B(\cdot,\cdot)\mid_{\mathfrak{m}_1\oplus\mathfrak{m}_2\oplus \cdots \oplus \mathfrak{m}_s}, \ (\lambda >0),
 $$
where  $\Lambda$ is the  operator associated to the metric $g$.
\end{theorem}

\begin{corollary}\label{CC1}
Let  $G/K$ be a generalized flag manifold  with $s\geq 3 $ in the decomposition  (\ref{2}). Let $G/K_1$ be the corresponding $M$-space. If  $\dim{\mathfrak{s}}=1$ and there exists  some $j\in  \{1,\dots, s\}$  such that $ \mathfrak{m}_{j}$ is reducible as an $\Ad(K_1)$-module, then $(G/K_1, g)$  is a  g.o. space if and only if $g$ is the standard metric.
\end{corollary}

For a generalized flag manifold with two isotropy summands we always assume  that it satisfies $[\mathfrak{m}_1, \mathfrak{m}_1]\subseteq \mathfrak{k}\oplus\mathfrak{m}_2$.  So it does not occur $\dim\mathfrak{m}_1=\dim\mathfrak{m}_2=2,$  since $G$ is simple. If there exists $i\in \{1, 2\}$ such that $\dim\mathfrak{m}_i=2$, we obtain that $i=2$ (see also convention on [Ar-Ch]).

 \begin{theorem} \label{T2}  Let  $G/K$ be a generalized flag manifold  with $s=2$ in the decomposition  (\ref{2}). Let $G/K_1$ be the corresponding $M$-space.

 1) If   both $ \mathfrak{m}_1$ and $ \mathfrak{m}_2$ are irreducible as $\Ad(K_1)$-modules, then $(G/K_1, g)$  is a  g.o. space if and only if for every $V \in \mathfrak{s}$ and for every $ X=X_1+X_2\in \mathfrak{m}$ ($X_i \in \mathfrak{m}_i, i=1, 2$),  there exists $k\in \mathfrak{k}_1$  such that
$$
   [\mu_1 k +(\mu_1-\mu)V+(\mu_1-\mu_2)X_2, X_1]=0\  \ \mbox{and}
$$
$$
   [\mu_2 k +(\mu_2-\mu)V, X_2]=0.
$$
Here $g=\langle \cdot,\cdot\rangle=\mu\mid_{\mathfrak{s}}+\mu_1B(\cdot,\cdot)\mid_{\mathfrak{m}_1}+\mu_2B(\cdot,\cdot )\mid_{\mathfrak{m}_2}$
is any $G$-invariant metric on $G/K_1$.

2) If there exists $j \in \{1, 2\}$ such that $\mathfrak{m}_j$ is reducible as an $\Ad(K_1)$-module and $\mathfrak{m}_i$ ($i\ne j$) is irreducible as  an $\Ad(K_1)$-module,  and   $(G/K_1, g)$  is  a  g.o. space, then

  \begin{equation}\label{eq11}
  g=\langle \cdot,\cdot\rangle=\mu B(\cdot,\cdot)\mid_{\mathfrak{s}+\mathfrak{m}_j}+\mu _i B(\cdot,\cdot)\mid_{\mathfrak{m}_i}.
  \end{equation}

 3) If  both $\mathfrak{m}_1$  and   $\mathfrak{m}_2$ are  reducible as  $\Ad(K_1)$-modules,  then   $(G/K_1, g)$  is  a  g.o. space if and only if $g$ is the standard metric.
\end{theorem}

\begin{corollary}\label{C2}
 Let $G/K$ be a generalized flag manifold with $s=2$ and $\dim\mathfrak{m}_2=2$.
 Then the corresponding $M$-space $(G/K_1, g)$,  where  $g$ is given by (\ref{eq11}),   is a g.o. space for every $\mu  \neq \mu_1$  if and only if for every  $V\in \mathfrak{s}$ and for every $X_1\in \mathfrak{m}_1, X_2\in \mathfrak{m}_2$ there exists $k\in \mathfrak{k}_1$ such that
 \begin{equation}\label{37}
 [k+V+X_2, X_1]=0.
 \end{equation}
 \end{corollary}

 \begin{theorem} \label{T3}  Let  $G/K$ be a generalized flag manifold  with $s=1$ in the decomposition  (\ref{2}). Let $G/K_1$ be the corresponding $M$-space.

 1) If   $ \mathfrak{m}$  is irreducible as $\Ad(K_1)$-module, then g.o. metrics on $G/K_1$ have been studied by Z.Chen and Yu.Nikonorov\ $(\mathrm{cf.}$ \cite{Ch-Ni}).

2) If   $\mathfrak{m}$  is  reducible as  $\Ad(K_1)$-module,  then   $(G/K_1, g)$  is  a  g.o. space if and only if $g$ is the standard metric.
\end{theorem}

 The paper is organised as follows:  In Section 1 we  recall certain Lie theoretic properties of a generalized flag manifold $G/K$ and corresponding $M$-space $G/K_1$. Then we investigate   $\Ad(K_1)$-irreducible submodules (cf. Lemma \ref{L2}, \ref{Le3}). In Section 2, 3 we  recall  basic facts  about g.o. spaces. In Sections 4, 5, 6  we  give the proofs of Theorems \ref{T1}, \ref{T2},  \ref{T3} and Corollaries  \ref{CC1}, \ref{C2}.

 \medskip

\noindent
{\bf Acknowledgements.}
The first author was supported by Grant \# E.037 from the research committee of the
University of Patras (Programme K. Karatheodori).

The second author is supported by NSFC 11501390,
and by funding of Sichuan University of Science and Engineering grant
Grant No. 2015RC10.

The third author is supported by NSFC 11571242.

\noindent

\section{Generalized flag manifolds and $M$-spaces}\label{Section1}

Let  $G/K=G/C(S)$  be a  generalized flag manifold, where $G$  is a compact semisimple Lie group and $S$ is a torus in $G$, here $C(S)$ denotes the centralizer of $S$ in G. Let $\mathfrak{g}$ and $\mathfrak{k}$ be the Lie algebras of the Lie groups $G$ and $K$ respectively, and  $\mathfrak{g}^\mathbb{C}$ and $\mathfrak{k}^\mathbb{C}$ be the complexifications of $\mathfrak{g}$ and $\mathfrak{k}$ respectively. Let $\mathfrak{g}=\mathfrak{k}\oplus\mathfrak{m}$ be a reductive decomposition with respect to $B=-$Killing form  on $\mathfrak{g}$ with $[\mathfrak{k}, \mathfrak{m}]\subset \mathfrak{m}$. Let $T$ be a maximal torus of $G$ containing $S$. Then this is a maximal torus in $K$. Let $\mathfrak{a}$ be the Lie algebra of $T$ and $\mathfrak{a}^{\mathbb{C}}$ its complex. Then  $\mathfrak{a}^{\mathbb{C}}$ is a Cartan subalgebra of $\mathfrak{g}^{\mathbb{C}}$. Let $R$ be a root system of $\mathfrak{g}^{\mathbb{C}}$ with respect to $\mathfrak{a}^{\mathbb{C}}$  and  $\Pi=\{\alpha_1, \dots, \alpha_l\}, (l=\dim_{\mathbb{C}} \mathfrak{a}^\mathbb{C})$ be a system of simple roots of $R$, and  $\{\Lambda_1, \dots, \Lambda_l\}$ be the fundamental weights of $\mathfrak{g}^{\mathbb{C}}$ corresponding to $\Pi$, that is $\frac{2B(\Lambda_i, \alpha_j)}{B(\alpha_j, \alpha_j)}=\delta_{ij}, ( 1 \leq i, j\leq l)$. We  can identify $(\mathfrak{a}^\mathbb{C})^\ast$
with $\mathfrak{a}^\mathbb{C}$ as follows: For every
$\alpha\in(\mathfrak{a}^\mathbb{C})^\ast$ it corresponds to  $h_\alpha\in
\mathfrak{a}^\mathbb{C}$ by the equation $B(H, h_\alpha)=\alpha(H)$ for all
$H \in \mathfrak{a}^\mathbb{C}$. Let
\begin{equation}
\mathfrak{g}^{\mathbb{C}}=\mathfrak{a}^{\mathbb{C}}\oplus \sum_{\alpha \in R}{\mathfrak{g}^{\mathbb{C}}_{\alpha}}
\end{equation}
be the root space decomposition,  where
\begin{equation}
{\mathfrak{g}^{\mathbb{C}}_{\alpha}}=\{ X\in \mathfrak{g}^{\mathbb{C}}: [H, X]=\alpha(H)X\; \; \forall H\in \mathfrak{a}^{\mathbb{C}}\}.
\end{equation}

Since ${\mathfrak{k}}^{\mathbb{C}}$ contains $\mathfrak{a}^{\mathbb{C}}$, there is a subset $R_K$ of $R$ such that $\mathfrak{k}^{\mathbb{C}}=\mathfrak{a}^{\mathbb{C}}\oplus \sum_{\alpha \in R_K}{\mathfrak{g}^{\mathbb{C}}_{\alpha}}$. We choose a system of simple roots $\Pi_K$ of $R_K$ and a system of simple roots $\Pi$  of $R$ so that $\Pi_K\subset \Pi$.  We choose an ordering in $R^{+}$.  Then  there is a natural ordering in $R^{+}_K$,  so that $R^{+}_K\subset R^{+}$.     Set $R_M=R\setminus R_K$ (complementary roots). Then $\mathfrak{m}^{\mathbb{C}}=\sum_{\alpha\in R_{M}}\mathfrak{g}^{\mathbb{C}}$.

\begin{definition}\label{D1} An invariant ordering $R_{M}^{+}$  in  $R_{M}$ is a choice of a subset $R_{M}^{+} \subset R_{M}$  such that

\noindent
(i) $R=R_K \sqcup R_{M}^{+}\sqcup R_{M}^{-}$,  where $R_ {M}^-=\{ -\alpha: \alpha\in R_{M}^+ \}$,

\noindent
(ii) If $\alpha, \beta \in  R_{M}^+$  and  $\alpha+\beta\in  R_{M}$,   we have $ \alpha+\beta\in  R_{M}^+$,

\noindent
(iii) If $\alpha \in R_{M}^+, \beta \in R_{K}^+ $ and $\alpha+\beta \in R $,  we have $ \alpha+\beta \in R_{M}^+$.

For any $\alpha, \beta \in R_{M}^+$ we define $\alpha>\beta$ if and only if ${\alpha-\beta} \in R_{M}^+$.
\end{definition}

We choose a Weyl basis $\{E_{\alpha}, H_{\alpha}: \alpha \in R\}$ in $\mathfrak{g}^{\mathbb{C}}$ with $B(E_{\alpha}, E_{-\alpha})=1, [E_{\alpha}, E_{-\alpha}]=H_{\alpha}$ and
\begin{equation}\label{4}
[E_\alpha,E_\beta]= \begin {cases} 0, & if \,\  \alpha + \beta  \not \in R \; and \; \alpha + \beta\neq 0,\\
N_{ \alpha, \beta}E_{ \alpha+\beta}, & if \,\   \alpha+ \beta \in R, \end {cases}
\end{equation}
where the structural constants $N_{\alpha,\beta}\;(\neq 0)$ satisfy $N_{\alpha,\beta}=-N_{{-\alpha},-{\beta}} $ and $ N_{\beta,\alpha}=-N_{{\alpha},{\beta}}$. Then we have that
\begin{equation}
\mathfrak{g}^{\mathbb{C}}=\mathfrak{a}^{\mathbb{C}} \oplus\sum_{\alpha\in R_K}\mathfrak{g}^{\mathbb{C}}_{\alpha}\oplus\sum_{\alpha\in R_M}\mathfrak{g}^{\mathbb{C}}_{\alpha},
\end{equation}
and $\{E_{\alpha}: \alpha \in R_M\}$  is a basis of $\mathfrak{m}^{\mathbb{C}}$. It is well known that
\begin{equation}\label{6}
\mathfrak{g}_{u}=\sum_{\alpha \in R^{+}}{\mathbb{R}\sqrt{-1}}H_{\alpha}\oplus\sum_{\alpha\in R^+}(\mathbb{R}A_\alpha +\mathbb{R}B_\alpha),
\end{equation}
where $A_\alpha=E_\alpha-E_{-\alpha}, B_\alpha=\sqrt{-1}(E_\alpha+E_{-\alpha})\ (\alpha\in R^+)$ is a  compact real form of $\mathfrak{g}^{\mathbb{C}}$.
Hence we can identify $\mathfrak{g}$  with $\mathfrak{g}_{u}$. In fact $\mathfrak{g}=\mathfrak{g}_{u}$ is the fixed point set of the conjugation
 ${X +\sqrt{-1}Y} \mapsto \overline{X +\sqrt{-1}Y}=X -\sqrt{-1} Y $ in  $\mathfrak{g}^\mathbb{C}$  so that $\overline{E_{\alpha}}=-E_{-\alpha}$. Hence $\mathfrak{k}=\sum_{\alpha \in R^{+}}{\mathbb{R}\sqrt{-1}}H_{\alpha} \oplus\sum_{\alpha\in R_{K}^+}(\mathbb{R}A_\alpha+\mathbb{R}B_\alpha).$
We set $R^{+}_M=R^{+}\setminus R^{+}_K$. Then
\begin{equation} \label{7}
\mathfrak{m}=\sum_{\alpha \in R_{M}^+}(\mathbb{R}A_\alpha+\mathbb{R}B_\alpha).
\end{equation}

The next lemma gives us  information about the Lie algebra structure of $\mathfrak{g}$.

\begin{lemma}\label{L1}
The Lie bracket among the elements of $ \{A_\alpha, B_\alpha, \sqrt{-1}H_\beta:\; \alpha \in {R}^+, \, \beta \in \Pi \}$ of $\mathfrak{g}$ are given by
 \begin{eqnarray*}
&&[\sqrt{-1}H_\alpha, A_\beta]=\beta(H_\alpha)B_\beta ,  \quad \quad \quad \quad
 [A_\alpha,A_\beta]=N_{\alpha,\beta}A_{\alpha+\beta}+N_{-\alpha,\beta} A_{\alpha-\beta} \;(\alpha\neq \beta), \\
&& [\sqrt{-1}H_\alpha, B_\beta]=-\beta(H_\alpha)A_\beta, \quad \quad\quad [B_\alpha,B_\beta]=-N_{\alpha,\beta}A_{\alpha+\beta}-N_{\alpha,-\beta}A_{\alpha-\beta}\;(\alpha\neq \beta),\\
&&[A_\alpha, B_\alpha]=2\sqrt{-1}H_\alpha, \quad\quad\quad\quad\quad\quad[A_\alpha,B_\beta]=N_{\alpha,\beta}B_{\alpha+\beta}+N_{\alpha,-\beta}B_{\alpha-\beta}\; (\alpha\neq \beta),
 \end{eqnarray*}
where $N_{\alpha,\beta}$ are the structural constants in (\ref{4}).
\end{lemma}

An important invariant of a generalized flag manifold $G/K$ is the set $R_\mathfrak{t}$ of  $\mathfrak{t}$-roots.  Their importance arises from the fact that the knowledge of $R_\mathfrak{t}$ gives us crucial information about the decomposition of the isotropy representation of the flag manifold $G/K$.

From now on we fix a system of simple roots $\Pi=\{\alpha_{1},\dots ,\alpha_{r},\phi_{1},\dots ,\phi_{k}\}$ of $R$, so that $\Pi_K=\{\phi_{1},\dots ,\phi_{k}\}$ is a basis of the root system $R_K$ and $\Pi_M=\Pi\setminus\Pi_{K}=\{\alpha_{1}, \dots ,\alpha_{r}\}$ $(r+k=l)$.

We consider the decomposition $R=R_K\cup R_M$ and let
\begin{equation}\label{8}
\mathfrak{t}=\mathfrak{z}(  \mathfrak{k}^\mathbb{C}) \cap\sqrt{ -1}\mathfrak{a}=\{X\in \sqrt{-1}\mathfrak{a} :\phi(X)=0, \ \mathrm{for\;  all} \ \phi \in R_K\},
\end{equation}
where  $\mathfrak{z}(  \mathfrak{k}^\mathbb{C})$ is the center of $\mathfrak{k}^\mathbb{C}$.
Consider the  restriction map $\kappa: (\mathfrak{a}^{\mathbb{C}})^\ast\rightarrow \mathfrak{t}^\ast$ defined by $\kappa(\alpha)=\alpha\mid_\mathfrak{t}$, and set
$R_\mathfrak{t}=\kappa(R)=\kappa(R_M)$. Note that $\kappa(R_K)=0$ and $\kappa(0)=0.$

 The elements of $R_\mathfrak{t}$ are called {\it $\mathfrak{t}$-roots}. For an invariant ordering $R_{M}^+=R^+\setminus R_{K}^+$ in $R_{M}$, we set $R_{\mathfrak{t}}^+=\kappa(R_{M}^+)$ and $R_{\mathfrak{t}}^-=-R_{\mathfrak{t}}^+= \{-\xi: \xi \in R_{\mathfrak{t}}^+\}$. It is obvious that  $R_{\mathfrak{t}}^-=\kappa(R_{M}^-)$, thus the splitting  $R_{\mathfrak{t}}=R_{\mathfrak{t}}^-\cup R_{\mathfrak{t}}^+$ defines an ordering in  $R_{\mathfrak{t}}$. A $\mathfrak{t}$-root $\xi \in  R_{\mathfrak{t}}^+$  (respectively $\xi \in R_{\mathfrak{t}}^-$) will be called {\it positive} (respectively negative).
A $ \mathfrak{t} $-root is called {\it simple} if it is not a sum of two positive \ $ \mathfrak{t} $-roots.
The set $\Pi_\mathfrak{t}$ of all simple $\mathfrak{t}$-roots is called a {\it $\mathfrak{t}$-basis }of $\mathfrak{t}^\ast$, in the sense that any $\mathfrak{t}$-root can be written as a linear combination of its elements with integer coefficients of the same sign.

\begin{definition}[\cite{Al-Ar}] \label{D2}
(1)  Two $\mathfrak{t}$-roots $\xi, \eta \in R_{\mathfrak{t}}$ are called {\it adjacent} if one of the following occurs:

(i)  If $\eta$ is a multiple of $\xi$, then  $\eta\neq \pm 2\xi$  and $\xi\neq \pm2\eta$.

(ii) If $\eta$ is not a multiple of $\xi$, then $\xi+\eta \in R_{\mathfrak{t}}$  or $\xi-\eta\in R_{\mathfrak{t}}$.

(2) Two $\mathfrak{t}$-roots $\xi, \eta \in R_{\mathfrak{t}}$ are called {\it connected} if there is a chain of $\mathfrak{t}$-roots
$$
\xi=\xi_1, \xi_2, \dots, \xi_k=\eta
$$
such that $\xi_i, \xi_{i+1}$ are adjacent $(i=1, \dots,  k-1)$.
\end{definition}

We remark that $\xi$ and $\pm \xi$ are connected, and if $\xi, 2\xi$ are the only positive $\mathfrak{t}$-roots, then these are not connected. We define the relation
\begin{equation}
\xi \sim \eta \Leftrightarrow  \xi ,  \eta \; \mathrm{are\;  connected. }
\end{equation}
One can easily check that this is an equivalence relation. Let $R^{i}$ be the equivalent
classes consisting of mutually connected $\mathfrak{t}$-roots. Then the set $R_{\mathfrak{t}}$ is decomposed
into a disjoint union
\begin{equation}
R_{\mathfrak{t}} = R^{1} \cup \cdots \cup  R^{r}.
\end{equation}

\begin{definition}
The set of $\mathfrak{t}$-roots  $R_{\mathfrak{t}}$ is called connected if $r=1$.
\end{definition}

Also,  if $G$ is simple, then for $s \geq 3$ in the decomposition (\ref{2}), the set of $\mathfrak{t}$-roots is connected (cf. \cite{Al-Ar}).

\begin{proposition}  [\cite{Al-Pe}]  \label{P1}  There is one-to-one correspondence between $\mathfrak{t}$-roots and complex irreducible
$\mathrm{ad}(\mathfrak{k}^\mathbb{C})$-submodules $\mathfrak{m}_{\xi}$ of $\mathfrak{m}^\mathbb{C}$. This correspondence is given by
$$R_{\mathfrak{t}}\ni \xi \leftrightarrow \mathfrak{m}_{\xi} =\sum_{{\alpha \in R_{M}: \kappa(\alpha)=\xi}}{\mathbb{C}E_{\alpha}}.$$
Thus $\mathfrak{m}^\mathbb{C}=\sum_{{\xi \in R_{\mathfrak{t}}}}\mathfrak{m}_{\xi}$. Moreover, these submodules are inequivalent as
$\mathrm{ad}(\mathfrak{k}^\mathbb{C})$-modules.
\end{proposition}

Since the complex conjugation $ \tau: \mathfrak{g}^\mathbb{C} \rightarrow  \mathfrak{g}^\mathbb{C}$, $X +\sqrt{-1}Y \mapsto  X -\sqrt{-1} Y \; ( X,Y \in \mathfrak{g})$ of $\mathfrak{g}^\mathbb{C}$ with respect to the compact real form $\mathfrak{g}$ interchanges the root spaces, i.e. $\tau(E_{\alpha})=E_{-\alpha}$ and $\tau(E_{-\alpha})=E_{\alpha}$, a decomposition of the real $\Ad(K)$-module $\mathfrak{m}=(\mathfrak{m}^\mathbb{C})^\tau$ into real irreducible $\Ad(K)$-submodule is given by
\begin{equation}\label{9}
\mathfrak{m}=\sum_ {{{ \xi \in R_{\mathfrak{t}}^+}=\kappa(R_{M}^+)}} (\mathfrak{m}_\xi \oplus \mathfrak{m}_{-\xi})^\tau,
\end{equation}
where $V^\tau$ denotes the set of fixed points of the complex conjugation $\tau$ in a vector subspace $V\subset \mathfrak{g}^\mathbb{C}$. If, for simplicity, we set $R_{\mathfrak{t}}^+=\{\xi_1,\dots ,\xi_s \}$, then according to (\ref{9}) each real irreducible $\Ad(K)$-submodule $\mathfrak{m}_{i}=(\mathfrak{m}_{\xi_i}\oplus \mathfrak{m}_{-\xi_i})^\tau \ (1 \leq i \leq s)$  corresponding to the positive $\mathfrak{t}$-roots $\xi_i$, is given by
\begin{equation}\label{10}
\mathfrak{m}_i=\sum_ {{\alpha \in R_{M}^+:\kappa (\alpha)=\xi _i}} (\mathbb{R} A_\alpha +\mathbb{R} B_\alpha ).
\end{equation}

\begin{definition} [\cite{Wan}]
Let $G/K$ be a generalized flag manifold with  $G$ simple and $K=C(S)=S\times K_1$, where $S$ is a torus in $G$ and $K_1$  is the semisimple part of $K$, here $C(S)$ denotes the centralizer of $S$ in G.
The homogeneous space $G/K_1$ is called the corresponding $M$-space.
\end{definition}

The following Lemma is important in our study.

\begin{lemma}\label{L2}
Assume that an  $\Ad(K)$-irreducible submodule $\mathfrak{m}_{i}\ (i\in \{1, \dots,  s \})$ is  $\Ad(K_1)$-reducible. Then we have a decomposition   $\mathfrak{m}_{i}=\mathfrak{n}^i_1\oplus  \mathfrak{n}^i_{2}$, where $\mathfrak{n}^i_{1}$ and   $\mathfrak{n}^i_{2}$  are equivalent irreducible $\Ad(K_1)$-invariant  submodules.
\end{lemma}

\begin{proof}
Let $R_{\mathfrak{t}}^+=\{\xi_1,\dots ,\xi_s \}$ be a set of positive $\mathfrak{t}$-roots for  the  generalized flag manifold $G/K$.
Set
\begin{equation}\label{w2}
R_{j}^+= \{ \alpha \in R_M^+: \; \alpha\mid_{\mathfrak{t}}=\xi _j \},  \; j=1,\dots, s.
\end{equation}

   Set $\alpha _j|_{\mathfrak{t}}=\overline{\alpha_j}, \, (j=1, \dots, r)$, where $\Pi_M=\Pi\setminus\Pi_K=\{\alpha_1, \dots, \alpha_r\}$. Then $\Pi_{\mathfrak{t}}=\{\overline{\alpha_1}, \dots, \overline{\alpha_r}\}$ is a $\mathfrak{t}$-basis of $\mathfrak{t}^{\ast}$. Therefore,   there exist  $\alpha_{i_1}, \dots, \alpha_{i_p}, i_1, \dots, i_p \in \{1, \dots, r\}$ and positive integers $b_1, \dots, b_p$  such that $\alpha\mid_{\mathfrak{t}}=b_1\overline{\alpha_{i_1}}+\cdots+b_p\overline{\alpha_{i_p}}$  for any $\alpha \in R_{i}^+$.  It follows that $B(\Lambda_{i_j}, \alpha)\neq 0$  for any  $\alpha\in R_{i}^+$ and $j\in \{1, \dots, p\}$. Then we have that
\begin{equation}\label{equ.1}
[\sqrt{-1}h_{{\Lambda_{i_j}}}, A_{\alpha}]=B(\Lambda_{i_j}, \alpha)B_{\alpha},\quad [\sqrt{-1}h_{{\Lambda_{i_j}}}, B_{\alpha}]=-B(\Lambda_{i_j}, \alpha)A_{\alpha}
\end{equation}
for any  $\alpha\in R_{i}^+$.

Assume that an  $\Ad(K)$-irreducible submodule $\mathfrak{m}_{i}\ (i\in \{1, \dots,  s \})$ is  $\Ad(K_1)$-reducible and set $\mathfrak{m}_{i}=\mathfrak{n}^i_1\oplus \cdots \oplus \mathfrak{n}^i_{l_i}$, where $\mathfrak{n}^i_{k}\; (k=1,\dots, l_i)$ are  irreducible $\Ad(K_1)$-invariant  submodules.  By (\ref{equ.1}) it is easy to check that $[\sqrt{-1}h_{\Lambda_{i_1}}, \mathfrak{n}^i_{1}]=[\sqrt{-1}h_{\Lambda_{i_q}}, \mathfrak{n}^i_{1}]$ for any $q\in \{2, \dots, p\}$. Set $\mathfrak{n}^{'}=[\sqrt{-1}h_{\Lambda_{i_1}}, \mathfrak{n}^i_{1}]$. Then by (\ref{equ.1}) it follows that  $\mathfrak{n}^i_{1}=[\sqrt{-1}h_{\Lambda_{i_1}}, [\sqrt{-1}h_{\Lambda_{i_1}}, \mathfrak{n}^i_{1}]]=[\sqrt{-1}h_{\Lambda_{i_1}}, \mathfrak{n}^{'}]=[\sqrt{-1}h_{\Lambda_{i_q}}, \mathfrak{n}^{'}]$ for any $q\in \{2, \dots, p\}$.  Then we have  that $\mathfrak{n}^i_{1}$ and $  \mathfrak{n}^{'}$ are irreducible $\Ad(K_1)$-invariant subspaces which are isomorphic. Since $\Pi_M=\{\alpha_1, \dots, \alpha_r\}$, this implies that  $\mathfrak{s}=\{H\in \mathfrak{a}: B(H, \Pi_K)=0 \}=\mathbb{R}\sqrt{-1}h_{\Lambda_1}+\cdots+\mathbb{R}\sqrt{-1}h_{\Lambda_r}$. Since  $ \mathfrak{k}=\mathfrak{s}\oplus \mathfrak{k}_1$, it follows that $[\mathfrak{s}\oplus \mathfrak{k}_1, \mathfrak{n}^i_{1}]=[\mathfrak{s}, \mathfrak{n}^i_{1}]+[\mathfrak{k}_1, \mathfrak{n}^i_{1}]\subset \mathfrak{n}^{'}+\mathfrak{n}^i_{1}$
and $[\mathfrak{s}\oplus \mathfrak{k}_1, \mathfrak{n}^{'}]=[\mathfrak{s}, \mathfrak{n}^{'}]+[\mathfrak{k}_1, \mathfrak{n}^{'}]\subset \mathfrak{n}^i_{1}+\mathfrak{n}^{'}$,  so $\mathfrak{n}^i_{1}+ \mathfrak{n}^{'}$ is an $\Ad(K)$-invariant subspace of $\mathfrak{m}_i$. Since $\mathfrak{m}_i$ is irreducible as an $\Ad(K)$-module, we obtain that $\mathfrak{m}_i=\mathfrak{n}^i_{1}+ \mathfrak{n}^{'}$. This implies that $\mathfrak{n}^i_2\oplus \cdots \oplus \mathfrak{n}^i_{l_i}\subseteq \mathfrak{n}^{'}$. Since $\mathfrak{n}^i_{1}$ and $  \mathfrak{n}^{'}$  are isomorphic, it follows that $\dim\mathfrak{n}^i_1\geq \dim\mathfrak{n}^i_2$. In the same way we obtain that $\dim\mathfrak{n}^i_2\geq \dim\mathfrak{n}^i_1$.  Hence we have $\mathfrak{n}^{'}=\mathfrak{n}^i_2$, which implies that  $l_i=2$. It is obvious that $\mathfrak{n}^i_{1}$ and   $\mathfrak{n}^i_{2}$  are equivalent by the map $\ad(\sqrt{-1}h_{\Lambda_{i_1}}): \mathfrak{n}^i_{1}\rightarrow \mathfrak{n}^{'}=\mathfrak{n}^i_{2}$.
\end{proof}

\begin{remark}\label{R1}
There are several choices for the equivalent map from   submodule $\mathfrak{n}^i_{1}$ to $\mathfrak{n}^i_{2}$
 in Lemma \ref{L2}, it is $\ad(\sqrt{-1}h_{\Lambda_{i_q}}): \mathfrak{n}^i_{1}\rightarrow \mathfrak{n}^i_{2}$.
  for any $q \in \{1, \dots, p\}$.
    Here it is  $\alpha\mid_{\mathfrak{t}}=a_1\overline{\alpha_{i_1}}+\cdots+a_p\overline{\alpha_{i_p}}$  for any $\alpha \in R_{i}^+$,
     where $i_1, \dots, i_p \in \{1, \dots, r\}$ and
      $\alpha _j|_{\mathfrak{t}}=\overline{\alpha_j}, \; (j=1, \dots, r)$.
\end{remark}

\begin{remark}\label{R2}
 If there exists $j\in \{1, \dots, s\}$ such that $\dim \mathfrak{m}_j=2$, then  $\mathfrak{m}_{j}$ is  $\Ad(K_1)$-reducible, and $\mathfrak{m}_{j}=\mathfrak{n}^j_1\oplus \mathfrak{n}^j_{2}$, where $\mathfrak{n}^j_{1}=\mathbb{R}A_{\alpha}$ and $\mathfrak{n}^j_{2}=\mathbb{R}B_{\alpha},  \; (\alpha \in R_{j}^+)$.
\end{remark}

With regard to the set (\ref{w2}) let $R_i^- = -R_i^+$ and let $R_i = R_i^+\sqcup R_i^-$. Also, we fix a Cartan subalgebra $\mathfrak{a}_1$ of the Lie algebra $\mathfrak{k}_1$.

\begin{definition}
A root $\alpha\in R_i$ is called the lowest (resp. highest) root in $R_i^+$ if $\alpha - \gamma\notin R$ (resp. $\alpha + \gamma\notin R$) for any $\gamma\in R_K^+$.
\end{definition}

\begin{remark}\label{R3}
 If  $\alpha, \beta \in R^+_i$ are the lowest and highest roots respectively,  $\alpha=\beta$ if and only if  $\dim{\mathfrak{m}_i}=2$.
\end{remark}

The following Propositions and Lemma are very important for the proofs  of Theorems \ref{T1}, \ref{T2},  \ref{T3} and Corollaries  \ref{CC1}, \ref{C2}.
\begin{proposition}\label{PP3}
 Let $\alpha^{'}, \beta^{'} $ be the lowest and highest root respectively in $R_i^+$, for some $i=1, \dots , s$.  If there exist $\alpha, \beta\in R_i$ such that
such that $\alpha|_{\mathfrak{a}_1}=\beta|_{\mathfrak{a}_1}$ and $\alpha(h)=-\beta(h)$ for any $h \in \mathfrak{s}$, then we have  $\alpha^{'}|_{\mathfrak{a}_1}=-\beta^{'}|_{\mathfrak{a}_1}$ and $\alpha^{'}(h)=\beta^{'}(h)$ for any $h \in \mathfrak{s}$. Also, we have that  $\alpha^{'}=\beta^{'}$ if and only if   $\beta=-\alpha$.
\end{proposition}
\begin{proof}
  Since $\alpha|_{\mathfrak{a}_1}=\beta|_{\mathfrak{a}_1}$, it follows that $\alpha+\delta \in R$ if and only if $\beta+\delta \in R$ for any $\delta \in R_K$. Since  $\alpha(h)=-\beta(h)$ for any $h \in \mathfrak{s}$,  it follows that one of $\alpha, \beta$ belongs to $R^+_i$, the other belongs to $R^-_i$.

  If   $\alpha \in R^+_i$ and  $\alpha$  is the highest root, it follows that $-\beta$ is the lowest root.
 We set $\alpha^{'}=\alpha$ and $\beta^{'}=-\beta$, it is easy to check that the conclusion holds.

 If   $\alpha \in R^+_i$ and  $\alpha$  is not the highest root,
 then there exists $\gamma \in R^+_K$ such that $\alpha+\gamma \in R^+_i$. Since $(\alpha+\gamma)(h)=\alpha(h)$ and $(\beta+\gamma)(h)=\beta(h)$ for any $h\in \mathfrak{s}$, we have $(\alpha+\gamma)(h)=-(\beta+\gamma)(h)$. Since $\alpha+\gamma \in R^+_i$,  it follows that $\beta+\gamma \in R^-_i$. Set $\alpha+\gamma=\beta_1$ and $\beta+\gamma=\omega_1$. If $\beta_1$  is not the highest root, then there exists $\gamma_1 \in R_K$ such that $\beta_1+\gamma_1 \in R^+_i $ and $\omega_1+\gamma_1 \in R^-_i$. We do this several times, until we get $\beta_k=\beta_{k-1}+\gamma_{k-1}$,  where  $\beta_k\in R^+_i$ and $\beta_k+\gamma^{'} \notin R$ for any $\gamma^{'}\in R^+_K$. This implies that $\beta_k$ is the highest root in $R^+_i$.  Then it follows that $\omega_k=(\omega_{k-1}+\gamma_{k-1})\in R^-_i$ and $\omega_k+\gamma^{'} \notin R$ for any $\gamma^{'}\in R^+_K$. This implies that $-\omega_k \in R^+_i$  is the lowest root. Since $\beta_k=\alpha+\gamma+\gamma_1+\cdots+\gamma_{k-1}$ and $\omega_k=\beta+\gamma+\gamma_1+\cdots+\gamma_{k-1}$, it follows that $\beta_k|_{\mathfrak{a}_1}=\omega_k|_{\mathfrak{a}_1}$.  Since $\beta_k(h)=(\alpha+\gamma+\gamma_1+\cdots+\gamma_{k-1})(h)=\alpha(h)$ and $\omega_k(h)=(\beta+\gamma+\gamma_1+\cdots+\gamma_{k-1})(h)=\beta(h)$ for any $h\in \mathfrak{s}$, it follows that $\beta_k(h)=-\omega_k(h)$ for any $h\in \mathfrak{s}$.

 Next we prove that $\alpha^{'}=\beta^{'}$ if and only if $\alpha=-\beta$.

 If  $\alpha=-\beta$,  since $\alpha\mid_{\mathfrak{a}_1}=\beta\mid_{\mathfrak{a}_1}$, it follows that $\alpha\mid_{\mathfrak{a}_1}=\beta\mid_{\mathfrak{a}_1}=0$, which implies that $\dim \mathfrak{m}_i=2$. Then it follows that  $\alpha^{'}=\beta^{'}$.

 If $\alpha^{'}=\beta^{'}$, by Remark \ref{R3} we have  $\dim \mathfrak{m}_i=2$. Since $\alpha(h)=-\beta(h)$ for any $h\in \mathfrak{s}$, it follows that $\alpha= -\beta$.
\end{proof}

\begin{lemma}\label{Le3}
An  $\Ad(K)$-irreducible submodule $\mathfrak{m}_{i}\ (i\in \{1, \dots,  s \})$ is  $\Ad(K_1)$-reducible given by $\mathfrak{m}_i = \mathfrak{n}_1^i \oplus\mathfrak{n}_2^i$ if and only if $\alpha|_{\mathfrak{a}_1}=-\beta|_{\mathfrak{a}_1}$ and $\alpha(h)=\beta(h)$ for any $h \in \mathfrak{s}$, where $\alpha, \beta \in R^+_i$ are the lowest and highest roots respectively, and  $\mathfrak{a}_1$ is  the Cartan subalgebra of $\mathfrak{k}_1$. Moreover, if $\dim \mathfrak{m}_i\neq 2$ and $\mathfrak{m}_i$ is  reducible as an $\Ad(K_1)$-submodule, we have $\mathfrak{m}_i=\mathfrak{n}^i_1\oplus\mathfrak{n}^i_2$ and
\begin{equation}\label{L34}
\begin {cases} \mathfrak{n}^i_1=U(\mathfrak{k}_1)(A_{\alpha}+A_{-\beta}), \\
\mathfrak{n}^i_2=U(\mathfrak{k}_1)(A_{\alpha}-A_{-\beta}), \end {cases}
\end{equation}
 where $U(\mathfrak{k}_1)$ is the universal enveloping algebra of $\mathfrak{k}_1$ and the action of $U(\mathfrak{k}_1)$
on $\mathfrak{n}^i_{1}$ and   $\mathfrak{n}^i_{2}$ is determined by the adjoint representation of $\mathfrak{g}$.
\end{lemma}
\begin{proof}
 We prove the necessity of  Lemma \ref{Le3} first. If $ {\mathfrak{m}_i}$ is reducible as $\Ad(K_1)$-module and $\dim {\mathfrak{m}_i}=2$,  Remark \ref{R3} implies  that $\beta=\alpha$.  It is obvious that the results hold.

       Let $ {\mathfrak{m}_i}$ be reducible as $\Ad(K_1)$-module and  $\dim {\mathfrak{m}_i}\neq2$.
      Assume that $\alpha|_{\mathfrak{a}_1}\neq\beta|_{\mathfrak{a}_1}$ for any $\alpha\neq \beta \in R_i$. Then it follows that $U(\mathfrak{k}_1)(X)=\mathfrak{m}_i$ for any nonzero vector $X\in \mathfrak{m}_i$, which means that $\mathfrak{m}_i$ is irreducible as $\Ad(K_1)$-module and this is a contradiction. Hence we have there that exist $\alpha \neq \beta \in R_i$ such that $\alpha|_{\mathfrak{a}_1}=\beta|_{\mathfrak{a}_1}$. Since $\alpha,  \beta \in R_i$, it follows that $\alpha(h)=\beta(h)$ or $\alpha(h)=-\beta(h)$ for any $h\in \mathfrak{s}$. Since $\alpha\neq\beta $ we have $\alpha(h)=-\beta(h)$ for any $h\in \mathfrak{s}$. By Proposition \ref{PP3} we have $\alpha^{'}|_{\mathfrak{a}_1}=-\beta^{'}|_{\mathfrak{a}_1}$ and $\alpha^{'}(h)=\beta^{'}(h)$ for any $h \in \mathfrak{s}$, where $\alpha^{'}$ and $\beta^{'}$ are the lowest root and highest root respectively.

Conversely, assume that   $\alpha|_{\mathfrak{a}_1}=-\beta|_{\mathfrak{a}_1}$ and $\alpha(h)=\beta(h)$ for any $h \in \mathfrak{s}$, where $\alpha, \beta \in R^+_i$ are the lowest and highest roots respectively. If $\alpha=\beta$, by Remark \ref{R3} we obtain that $\dim \mathfrak{m}_i=2$,  and Remark \ref{R2} implies that $\mathfrak{m}_i$ is reducible as $\Ad(K_1)$-module. If $\alpha\neq\beta$,  we set
\begin{equation}\label{LLL3}
\begin {cases} \mathfrak{n}^i_1=U(\mathfrak{k}_1)(A_{\alpha}+A_{-\beta}), \\
\mathfrak{n}^i_2=U(\mathfrak{k}_1)(A_{\alpha}-A_{-\beta}). \end {cases}
\end{equation}

Since $\alpha|_{\mathfrak{a}_1}=-\beta|_{\mathfrak{a}_1}$, it follows that $\alpha+\gamma \in R$ if and only if $\beta-\gamma \in R$, where $\gamma \in R^+_K$. By Lemma \ref{L1} we have
\begin{equation}\label{LL3}
\begin{cases}
[\sqrt{-1}h_{\gamma}, A_{\alpha}\pm A_{-\beta}]=(\alpha(h_{\gamma})(B_{\alpha}\pm B_{-{\beta}}),\\
\\
[\sqrt{-1}h_{\gamma}, B_{\alpha}\pm B_{-\beta}]=-\alpha(h_{\gamma})(A_{\alpha}\pm A_{-{\beta}}),\\
\\
[A_{\gamma}, A_{\alpha}\pm A_{-\beta}]=N_{\gamma, \alpha}A_{\gamma+\alpha}\pm N_{\gamma, -\beta}A_{\gamma-\beta},\\
\\
[B_{\gamma}, A_{\alpha}\pm A_{-\beta}]=(-N_{\alpha, \gamma}B_{\alpha+\gamma})\pm(-N_{-\beta, \gamma}B_{-\beta+\gamma}),\\
\\
[B_{\gamma}, B_{\alpha}\pm B_{-\beta}]=(-N_{\gamma, \alpha}A_{\gamma+\alpha})\pm(-N_{\gamma, -\beta}A_{\gamma-\beta}),
\end{cases}
\end{equation}
where $\gamma\in R^+_K$.  If $\alpha +\gamma \notin R$, we have $N_{\alpha, \gamma}=0$, and if $\beta -\gamma \notin R$, we have $N_{\beta, -\gamma}=0$.  By (\ref{LL3})
it is easy to check that $\mathfrak{n}^i_1$ and $\mathfrak{n}^i_2$ are  $\Ad(K_1)$-invariant irreducible submodules,  and $\mathfrak{n}^i_{2}=[\sqrt{-1}h_{\Lambda_{i_q}}, \mathfrak{n}^i_{1}]$,
  $\mathfrak{n}^i_{1}=[\sqrt{-1}h_{\Lambda_{i_q}}, \mathfrak{n}^i_{2}]$
  for any $q \in \{1, \dots, p\}$.
    Here  $\alpha\mid_{\mathfrak{t}}=a_1\overline{\alpha_{i_1}}+\cdots+a_p\overline{\alpha_{i_p}}$  for any $\alpha \in R_{i}^+$,
     where $i_1, \dots, i_p \in \{1, \dots, r\}$ and
      $\alpha _j|_{\mathfrak{t}}=\overline{\alpha_j}, \; (j=1, \dots, r)$. Then it follows that $\mathfrak{n}^i_1+\mathfrak{n}^i_2$ is  $\Ad(K)$-invariant submodule. Since $\mathfrak{n}^i_1, \mathfrak{n}^i_2\subset \mathfrak{m}_i$ and $\mathfrak{m}_i$ is irreducible $\Ad(K)$-module, this implies that $\mathfrak{m}_i=\mathfrak{n}^i_1\oplus\mathfrak{n}^i_2$.
\end{proof}

\begin{proposition}\label{PP4}
Let   $\alpha^{'}|_{\mathfrak{a}_1}=-\beta^{'}|_{\mathfrak{a}_1}$ and $\alpha^{'}(h)=\beta^{'}(h)$ for any $h \in \mathfrak{s}$, where $\alpha^{'}, \beta^{'} \in R_i$ are the lowest root and highest root respectively. Then for any $\alpha \in  R^+_i$ there exists $\beta \in R^+_i$ such that $\alpha|_{\mathfrak{a}_1}=-\beta|_{\mathfrak{a}_1}$ and $\alpha(h)=\beta(h)$ for any $h \in \mathfrak{s}$.
\end{proposition}
\begin{proof}
 First we assume that $\alpha^{'} \neq\beta^{'}$.    Any $\alpha \in R_i$  can be expressed as $\alpha=\alpha^{'}+\gamma_0+\gamma_1+\cdots+\gamma_k$,  and $(\alpha^{'}+\gamma_0+\gamma_1+\cdots+\gamma_j)\in R^+_i,  (j\leq k)$.  If $k=0$, this means that $\alpha=\alpha^{'}$. Since $\alpha^{'}|_{\mathfrak{a}_1}=-\beta^{'}|_{\mathfrak{a}_1}$, it follows that $\alpha^{'}+\delta \in R$ if and only if $-\beta^{'}+\delta \in R$ for any $\delta \in R_K$. This implies that $\beta=-(-\beta^{'}+\gamma_0+\gamma_1+\cdots+\gamma_k)=(\beta^{'}-\gamma_0-\gamma_1-\cdots-\gamma_k) \in R^+_i$.

  Now we prove that $\alpha|_{\mathfrak{a}_1}=-\beta|_{\mathfrak{a}_1}$ and $\alpha(h)=\beta(h)$ for any $h \in \mathfrak{s}$.

  Since $\alpha=\alpha^{'}+\gamma_0+\gamma_1+\cdots+\gamma_k$, it follows that $\alpha(h)=(\alpha^{'}+\gamma_0+\gamma_1+\cdots+\gamma_k)(h)=\alpha^{'}(h)+(\gamma_0+\gamma_1+\cdots+\gamma_k)(h)$ for any $h\in \mathfrak{a}_1$.  Since $\beta=\beta^{'}-\gamma_0-\gamma_1-\cdots-\gamma_k$, it follows that $\beta(h)=(\beta^{'}-\gamma_0-\gamma_1+\cdots-\gamma_k)(h)=\beta^{'}(h)-(\gamma_0+\gamma_1+\cdots+\gamma_k)(h)$ for any $h\in \mathfrak{a}_1$. Since $\alpha^{'}|_{\mathfrak{a}_1}=-\beta^{'}|_{\mathfrak{a}_1}$, it follows that $\alpha^{'}(h)=-\beta^{'}(h)$ for any $h\in \mathfrak{a}_1$, hence we have $\alpha(h)=-\beta(h)$ for any $h\in \mathfrak{a}_1$, which means that $\alpha|_{\mathfrak{a}_1}=-\beta|_{\mathfrak{a}_1}$.
  Since $\gamma(h)=0$  and $\alpha^{'}(h)=\beta^{'}(h)$ for any $h\in \mathfrak{s}$, it follows that $\alpha(h)=(\alpha^{'}+\gamma_0+\gamma_1+\cdots+\gamma_k)(h)=\alpha^{'}(h)=\beta^{'}(h)=
  (\beta^{'}-\gamma_0-\gamma_1-\cdots-\gamma_k)(h)=\beta(h)$ for any $h \in \mathfrak{s}$.

 If  $\alpha^{'}=\beta^{'}$, we have  $\alpha=\beta=\alpha^{'}$.
\end{proof}

\section{Invariant metrics on $M$-spaces}

 Let $\mathfrak{n}$ be the tangent space of a $M$-space $G/K_1$ with $\mathfrak{n}=\mathfrak{s}\oplus \mathfrak{m}$.
 Let $\{\sqrt{-1}h_{\Lambda_1}, \dots, \sqrt{-1}h_{\Lambda_r}\}$ be a basis of $\mathfrak{s}$ with respect to  $\Pi_M=\{\alpha_1, \dots, \alpha_r\}$.  There are two cases for  $\Ad(K_1)$-invariant irreducible decomposition of the tangent space $\mathfrak{n}$, based on properties of the isotropy representation
$\mathfrak{m}=\mathfrak{m}_1\oplus \cdots\oplus \mathfrak{m}_s$ of the flag manifold $G/K$. Then we consider $G$-invariant metrics on $G/K_1$ which are $\Ad(K_1)$-invariant corresponding to the $\Ad(K_1)$-irreducible decomposition
of $\mathfrak{n}$.

\smallskip
{\bf Case A.}  Assume that $\mathfrak{m}_i$ for any $i \in \{1, \dots, s\}$ in the decomposition (\ref{2}) is irreducible as an $\Ad({K}_1)$-submodule,  so we get that
 \begin{equation}\label{equ3}
 \mathfrak{n}=\mathbb{R}\sqrt{-1}h_{\Lambda_1}\oplus\cdots\oplus \mathbb{R}\sqrt{-1}h_{\Lambda_r}\oplus\mathfrak{m}_1\oplus\cdots\oplus \mathfrak{m}_s
 \end{equation}
 is the $\Ad(K_1)$-irreducible decomposition.

  Let  $\langle \cdot,\cdot \rangle=B(\Lambda\cdot,\cdot)$ be an $\mathrm{Ad}({K}_1)$-invariant scalar product on $\mathfrak{n}$, where  $\Lambda$ is the associated operator.
 Therefore,   $G$-invariant metrics on $G/K_1$ which are $\mathrm{Ad}({K}_1)$-invariant are defined by
 \begin{equation}\label{equ5}
\langle \cdot,\cdot \rangle=\Lambda|_{\mathfrak{s}}+ \lambda_1B(\cdot,\cdot)|_{{\mathfrak{m}}_1}+\cdots+ \lambda_s B(\cdot,\cdot)|_{{\mathfrak{m}}_s},
 \end{equation}
where $A|_{\mathfrak{s}}$ is positive definite symmetry matrix.

\smallskip
 {\bf Case B.} Assume that there exists an $r' \in \{1, \dots, s\}$ such that $ \mathfrak{m}_{i}$, $i=1, \dots, r'$ are $\Ad(K_1)$-reducible submodules, and $ \mathfrak{m}_{i}$ are  $\Ad(K_1)$-irreducible submodules for $i=r'+1, \dots, s$.

 Set $\mathfrak{m}_{i}=\mathfrak{n}^i_1\oplus \mathfrak{n}^i_{2}$, $i=1,\dots, r'$, where $\mathfrak{n}^i_{1}$ and $\mathfrak{n}^i_{2}$ are  equivalent and irreducible $\Ad(K_1)$-submodules.
It follows that
 \begin{equation}\label{equ6}
 \mathfrak{n}=\mathbb{R}\sqrt{-1}h_{\Lambda_1}\oplus\cdots\oplus \mathbb{R}\sqrt{-1}h_{\Lambda_r}\oplus(\mathfrak{n}^1_1\oplus \mathfrak{n}^1_{2})\oplus\cdots\oplus (\mathfrak{n}^{r'}_1\oplus \mathfrak{n}^{r'}_{2})\oplus\mathfrak{m}_{r'+1}\oplus\cdots\oplus \mathfrak{m}_{s}
 \end{equation}
 is the $\Ad(K_1)$-irreducible decomposition.

  Let  $\langle \cdot,\cdot \rangle=B(\Lambda\cdot,\cdot)$ be an $\mathrm{Ad}({K}_1)$-invariant scalar product on $\mathfrak{n}$, where  $\Lambda$ is the associated operator. We fix  basis $\mathcal{B}=\{\sqrt{-1}h_{\Lambda_i}, A_{\alpha}, B_{\alpha}, i=1,\dots, r, \alpha \in R^+_M\}$ adapted to the decomposition (\ref{equ6}).  Let $A$ and $A|_{\mathfrak{p}}$ be the matrix representation of $\Lambda$ and $\Lambda|_{\mathfrak{p}}$ respectively, where $\mathfrak{p}$ denotes a subspace of $\mathfrak{n}$.
 Then $G$-invariant metrics on $G/K_1$ which are $\mathrm{Ad}({K}_1)$-invariant are defined by

 \begin{equation}\label{equ7}
A=\left (
\begin{array}{ccccc}
A|_{\mathfrak{s}}& 0 & \dots & 0 \\
0 & A|_{\mathfrak{m}_1} &  & \\
 \vdots&\dots&\ddots& \\
 0 & \dots & \dots & A|_{\mathfrak{m}_s}
\end{array}
\right ),
\end{equation}
where $A|_{\mathfrak{s}}$ is a positive definite symmetric matrix, and $A|_{\mathfrak{m}_i}$, $i=1,\dots, r'$ has the form

 \begin{equation}\label{equ8}
A|_{\mathfrak{m}_i}=\left (
\begin{array}{cc}
\mu^i_1\mathrm{Id} |_{\mathfrak{n}^i_1}& A^i_{21}  \\
A^i_{12} & \mu^i_2\mathrm{Id}|_{\mathfrak{n}^i_2}
\end{array}
\right ),  \mu^i_1, \mu^i_2> 0.
\end{equation}
 The  block matrices $A^i_{12}$ and  $A^i_{21}$ correspond to  $\Ad(K_1)$-equivariant  maps $\phi_1: \mathfrak{n}^i_1\rightarrow \mathfrak{n}^i_2$ and $\phi_2: \mathfrak{n}^i_2\rightarrow \mathfrak{n}^i_1$ respectively.
 Moreover, the symmetry of $\Lambda$ implies that $A^i_{12}=A^i_{21}$. Consequently, for any vector $X_j\in \mathfrak{n}^i_j\subset \mathfrak{m}_i, j=1, 2$, it is

 \begin{equation}\label{equ8}
 \Lambda X_j=\Lambda|_{\mathfrak{m}_i}X_j=\mu^i_jX_j+A^i_{jk}X_j,\quad A^i_{jk}:\mathfrak{n}^i_j\rightarrow \mathfrak{n}^i_k, \;j\neq k \in \{1, 2\}.
 \end{equation}

Since $\mathfrak{m}_p$, $p=r'+1,\dots, s$ are irreducible as $\Ad(K_1)$-modules, it follows that $\Lambda|_{\mathfrak{m}_p}=\mu_p\mathrm{Id}|_{\mathfrak{m}_p}$, for some $\mu_p>0$.

\section{Riemannian g.o. spaces}

 Let $(M=G/K, g)$ be a homogeneous Riemannian manifold with  $G$  a compact connected semisimple Lie group. Let $\mathfrak{g}$ and $\mathfrak{k}$ be the Lie algebras of $G$ and $K$ respectively and
  $\mathfrak{g}=\mathfrak{k}\oplus \mathfrak{m}$
be a reductive decomposition.

\begin{definition} \label {D3}  A nonzero vector $X \in \mathfrak{g}$ is called a geodesic vector if the curve (\ref{1}) is a geodesic.

\end{definition}

\begin{lemma}[\cite{Ko-Va}]  \label{L3} A nonzero vector  $X \in \mathfrak{g}$ is a geodesic vector if and only if

\begin{equation}\label{12}
  \langle [X,Y]_{\mathfrak{m}},X_{\mathfrak{m}} \rangle =0
\end{equation}
for all $Y \in \mathfrak{m}$.
Here the subscript $\mathfrak{m}$ denotes the projection into $\mathfrak{m}$.
\end{lemma}

A useful description of homogeneous geodesics (\ref{1}) is provided by the following :
\begin{proposition} [\cite{Al-Ar}]  \label{P2}

 Let $(M=G/K, g)$ be a homogeneous Riemannian manifold and $\Lambda$ be the associated operator. Let $a\in \mathfrak{k}$ and $x \in \mathfrak{m}$. Then the following are equivalent:

 (1)\  The orbit $\gamma(t)=\mathrm{exp}t(a+x)\cdot o$ of the one-parameter subgroup $\mathrm{exp}t(a+x)$ through the point $o=eK$ is a geodesic of $M$.

 (2) \ $[a+x, \Lambda x] \in \mathfrak{k}$.

 (3) \ $\langle [a, x], y \rangle = \langle x, [x, y]_{\mathfrak{m}} \rangle \ \mbox {for all} \  y \in \mathfrak{m}$.

 (4) \ $\langle [a+x, y]_{\mathfrak{m}}, x \rangle=0 \  \mbox {for all} \ y \in \mathfrak{m}$.
 \end{proposition}

 An important corollary of Proposition \ref{P2} is the following:

 \begin{corollary} [\cite{Al-Ar}]  \label{C3}

 Let $(M=G/K, g)$ be a homogeneous Riemannian manifold. Then $(M=G/K, g)$ is a g.o. space if and only if for every $x\in \mathfrak{m}$  there exists an $a(x) \in \mathfrak{k}$  such that
 \begin{equation}\label{13}
 [a(x)+x, \Lambda x] \in \mathfrak{k}.
 \end{equation}
 \end{corollary}

 For later use we recall the following:

 \begin{proposition}{\rm(}\cite[Proposition 5]{Al-Ni}{\rm )}\label{P3}
Let $(M=G/H,g)$ be a compact g.o. space with associated operator $\Lambda$.
Let $X,Y \in \mathfrak{m}$ be eigenvectors $\Lambda$ with different eigenvalues $\lambda, \mu$. Then
 \begin{equation}\label{14}
[X, Y]= \frac{\lambda}{\lambda-\mu}[h, X]+\frac{\mu}{\lambda-\mu}[h, Y]
\end{equation}
for some $h\in \mathfrak{h}$.
\end{proposition}

\begin{proposition}{\rm(}\cite{Gor}, \cite{Tam}{\rm )}\label{P5}
Let $G$ be a connected semisimple Lie group and $H\supset K$ be  compact Lie subgroups in $G$.  Let  $M_{F}$ and $M_{C}$ be the  tangent spaces of $F=H/K$ and  $C=G/H$ respectively. Then the metric $g_{a, b}= a B\mid_{M_{F}}+b B\mid_{M_{C}}, (a, b \in \mathbb{R}^+)$ is a g.o. metric on $G/K$ if and only if  for any $v_F\in M_{F}$, $v_C \in M_{C}$ there exists $X\in \mathfrak{k}$ such that
\begin{equation*}
[X, v_F]=[X+v_F, v_C]=0.
\end{equation*}
\end{proposition}

 \section{Proof of Theorem 1 and Corollary 1}

Let $G/K$ be a generalized flag manifold with $K=C(S)=S\times K_1$, where $S$ is a torus in the simple compact Lie group $G$ and $K_1$ is the semisimple part of $K$. Then the corresponding $M$-space is $G/K_1$. We denote by $\mathfrak{g}$ and  $\mathfrak{k}$ the Lie algebras of $G$ and $K$ respectively. Let $B=-$Killing form.
Then the module $\mathfrak{m}$  decomposes into a direct sum of  $\Ad(K)$-invariant irreducible submodules pairwise orthogonal with respect to $B$ (cf. (\ref{2})).

\smallskip
\noindent
{\it Proof of Theorem \ref{T1}}.

{\bf Case 1.} Assume that $\mathfrak{m}_i$ for any $i \in \{1, \dots, s\}$ in the decomposition (\ref{2}) is irreducible as an $\Ad({K}_1)$-submodule.
 Then the tangent space $\mathfrak{n}\cong T_{o}(G/K_1)$ is decomposed into irreducible $\Ad({K}_1)$-invariant submodules:
 \begin{equation}\label{15}
 \mathfrak{n}=\mathbb{R}\sqrt{-1}h_{\Lambda_1}\oplus\cdots\oplus \mathbb{R}\sqrt{-1}h_{\Lambda_r}\oplus \mathfrak{m}_1\oplus\mathfrak{m}_2\oplus\cdots\oplus \mathfrak{m}_s.
 \end{equation}
Here  $\mathfrak{s}=\mathbb{R}\sqrt{-1}h_{\Lambda_1}\oplus\cdots\oplus \mathbb{R}\sqrt{-1}h_{\Lambda_r}$.  Therefore,  an $\Ad(K_1)$-invariant inner product $\langle \cdot,\cdot \rangle=B(\Lambda\cdot,\cdot)$ is expressed as
 \begin{equation}\label{17}
\langle \cdot,\cdot \rangle=\Lambda|_{\mathfrak{s}}+ \lambda_1B(\cdot,\cdot)|_{{\mathfrak{m}}_1}+\cdots+ \lambda_s B(\cdot,\cdot)|_{{\mathfrak{m}}_s},
 \end{equation}
 where  $\Lambda$  is the associate operator on $\mathfrak{n}$.

Let $R_{\mathfrak{t}}^+=\{\xi_1,\cdots,\xi_s \}$ be the set of positive $\mathfrak{t}$-roots of the generalized flag manifold $G/K$ with $s\geq 3$. Since $R_{\mathfrak{t}}^+$ is connected,   for any $\xi, \eta \in R_{\mathfrak{t}}^+$    there exists (without loss of generality) a chain of positive $\mathfrak{t}$-roots
\begin{equation}\label{18}
\xi=\zeta_1, \zeta_2, \dots, \zeta_k=\eta,
\end{equation}
where  $\zeta_i, \zeta_{i+1}$ are adjacent $(i=1,\dots, k-1)$.

We define the   subset $\{\mathfrak{m}_{i_1}, \mathfrak{m}_{i_2}, \dots, \mathfrak{m}_{i_k}\}$ of $\{\mathfrak{m}_{1}, \mathfrak{m}_{2}, \dots, \mathfrak{m}_{s}\}$ by
\begin{equation}\label{19}
\mathfrak{m}_{i_q}=\sum_ {{\alpha \in R_{M}^+:\kappa (\alpha)=\zeta _q}} (\mathbb{R} A_\alpha +\mathbb{R} B_\alpha ), \ (q=1,\dots,k).
\end{equation}
 Since $\zeta_q, \zeta_{q+1}$ $(q=1,\dots, k-1)$ are adjacent, then either  $\zeta_q +\zeta_{q+1} \in R_{\mathfrak{t}}$ or  $\zeta_{q+1} -\zeta_q \in R_{\mathfrak{t}}$.  Therefore, we have
\begin{equation*}\label{20}
[\mathfrak{m}_{i_q}, \mathfrak{m}_{i_{q+1}}]\subseteq \left(\sum_ {{\alpha \in R_{M}^+:\kappa (\alpha)=\zeta _{q}+\zeta_{{q+1}}}} (\mathbb{R} A_\alpha +\mathbb{R} B_\alpha ) \right)\oplus\left(\sum_ {{\alpha \in R_{M}:\kappa (\alpha)=\zeta_{{q+1}}-\zeta_{{q}}}} (\mathbb{R} A_\alpha +\mathbb{R} B_\alpha )\right).
\end{equation*}
 If  $\zeta _{q}+\zeta_{{q+1}}\notin R_{\mathfrak{t}}$, then $\sum_ {\alpha \in R_{M}^+:\kappa (\alpha)=\zeta _{{q+1}}+\zeta_{{q}}} (\mathbb{R} A_\alpha +\mathbb{R} B_\alpha )=\{0\}$.  If  $\zeta _{{q+1}}-\zeta_{{q}}\notin R_{\mathfrak{t}}$, then $\sum_ {\alpha \in R_{M}:\kappa (\alpha)=\zeta _{{q+1}}-\zeta_{{q}}} (\mathbb{R} A_\alpha +\mathbb{R} B_\alpha )=\{0\}$.  Since $\zeta _{q}+\zeta_{{q+1}}\neq \pm \zeta _{q}$,  $\zeta _{q}+\zeta_{{q+1}}\neq\pm \zeta_{{q+1}}$  and  $\zeta _{q}-\zeta_{{q+1}}\neq \pm \zeta _{q}$,  $\zeta _{q}-\zeta_{{q+1}}\neq\pm \zeta_{{q+1}}$, it follows that $$  \left(\sum_ {{\alpha \in R_{M}^+:\kappa (\alpha)=\zeta _{q}+\zeta_{{q+1}}}} (\mathbb{R} A_\alpha +\mathbb{R} B_\alpha ) \oplus\sum_ {{\alpha \in R_{M}:\kappa (\alpha)=\zeta_{{q+1}}-\zeta_{{q}}}} (\mathbb{R} A_\alpha +\mathbb{R} B_\alpha )\right)\cap \left(\mathfrak{m}_{i_q}\oplus \mathfrak{m}_{i_{q+1}}\right)=\{0\}.$$  Therefore we get that
\begin{equation}\label{21}
[\mathfrak{m}_{i_q}, \mathfrak{m}_{i_{q+1}}]\cap (\mathfrak{m}_{i_q}\oplus \mathfrak{m}_{i_{q+1}})=\{0\}.
\end{equation}

 Also, since $\zeta_q, \zeta_{q+1}$ are adjacent $(q=1,\dots, k-1)$ in (\ref{18}), there exist $X\in \mathfrak{m}_{i_q}, Y \in \mathfrak{m}_{i_{q+1}}$  eigenvectors of $\Lambda$ such that $[X,Y]\neq 0$. If we had that  $\lambda_{i_q}\neq \lambda_{i_{q+1}}$,  then   Proposition \ref{P3}  implies that $[X,Y]\subset \mathfrak{m}_{i_q}\oplus \mathfrak{m}_{i_{q+1}}$, which contradicts  (\ref{21}), hence  $\lambda_{i_q}= \lambda_{i_{q+1}}, (q=1,\dots, k-1)$. Since this is true for any $\xi, \eta \in R_{\mathfrak{t}}^+$  we obtain that $\lambda_1=\lambda_2=\cdots=\lambda_s$,  and the conclusion follows.

\medskip

\smallskip
\noindent
{\bf Case 2.} Assume that there exists an $r' \in \{1, \dots, s\}$ such that $ \mathfrak{m}_{i}\ (i=1, \dots, r')$ are  reducible as  $\Ad(K_1)$-submodules, and $ \mathfrak{m}_{i}$ are irreducible as $\Ad(K_1)$-submodules for $i=r'+1, \dots, s$.

 By Lemma \ref{L2} we have  $\mathfrak{m}_{i}=\mathfrak{n}^i_1\oplus \mathfrak{n}^i_{2},\; (i=1,\dots, r')$, where $\mathfrak{n}^i_{1}, \mathfrak{n}^i_{2}$ are  equivalent and irreducible as $\Ad(K_1)$-submodules.
It follows that
 \begin{equation}
 \mathfrak{n}=\mathbb{R}\sqrt{-1}h_{\Lambda_1}\oplus\cdots\oplus \mathbb{R}\sqrt{-1}h_{\Lambda_r}\oplus(\mathfrak{n}^1_1\oplus \mathfrak{n}^1_{2})\oplus\cdots\oplus (\mathfrak{n}^{r'}_1\oplus \mathfrak{n}^{r'}_{2})\oplus\mathfrak{m}_{r'+1}\oplus\cdots\oplus \mathfrak{m}_{s}
 \end{equation}
 is the $\Ad(K_1)$-irreducible decomposition.  Therefore  $G$-invariant metrics on $G/K_1$ are defined  by (\ref{equ7}).

  Since $R_{\mathfrak{t}}^+=\{\xi_1,\cdots,\xi_s \}$ is connected,   for any $\xi, \eta \in R_{\mathfrak{t}}^+$    there exists 
   a chain of positive $\mathfrak{t}$-roots
\begin{equation}\label{118}
\xi=\zeta_1, \zeta_2, \dots, \zeta_k=\eta,
\end{equation}
where  $\zeta_i, \zeta_{i+1}$ are adjacent $(i=1,\dots, k-1)$.

We define the   subset $\{\mathfrak{m}_{i_1}, \mathfrak{m}_{i_2}, \dots, \mathfrak{m}_{i_k}\}$ of $\{\mathfrak{m}_{1}, \mathfrak{m}_{2}, \dots, \mathfrak{m}_{s}\}$ by
\begin{equation}\label{119}
\mathfrak{m}_{i_q}=\sum_ {{\alpha \in R_{M}^+:\kappa (\alpha)=\zeta _q}} (\mathbb{R} A_\alpha +\mathbb{R} B_\alpha ), \ (q=1,\dots,k).
\end{equation}
 Since $\zeta_q, \zeta_{q+1}$ $(q=1,\dots, k-1)$ are adjacent, it is either  $\zeta_q +\zeta_{q+1} \in R_{\mathfrak{t}}$ or  $\zeta_{q+1} -\zeta_q \in R_{\mathfrak{t}}$.  There are three possibilities for $\mathfrak{m}_{i_q}, \mathfrak{m}_{i_{q+1}}$ as $\Ad(K_1)$-modules:

\medskip
 {\bf (a)} Both of $\mathfrak{m}_{i_q}, \mathfrak{m}_{i_{q+1}}$ are irreducible as $\Ad(K_1)$-modules.

 Since  $\mathfrak{m}_{i_q}, \mathfrak{m}_{i_{q+1}}$ are irreducible as $\Ad(K_1)$-modules, it follows that the associated operator $\Lambda|_{\mathfrak{m}_{i_q}}=\mu_{i_q}\mathrm{Id}|_{\mathfrak{m}_{i_q}}, \Lambda|_{\mathfrak{m}_{i_{q+1}}}=\mu_{i_{q+1}}\mathrm{Id}|_{\mathfrak{m}_{i_{q+1}}}, (\mu_{i_q}, \mu_{i_{q+1}}> 0)$. As  in  the proof in  {\bf Case 1} we obtain that $\mu_{i_q}=\mu_{i_{q+1}}$.

\medskip
  {\bf (b)} One of  $\mathfrak{m}_{i_q}, \mathfrak{m}_{i_{q+1}}$ is irreducible as $\Ad(K_1)$-module and  the other is reducible as $\Ad(K_1)$-module.
\medskip

{\bf (b1)} Assume that $\mathfrak{m}_{i_q}$ is reducible as $\Ad(K_1)$-module.

    By Lemma \ref{L2} it follows that $\mathfrak{m}_{i_q}=\mathfrak{n}^{i_q}_1\oplus \mathfrak{n}^{i_q}_2$, where $\mathfrak{n}^{i_q}_1, \mathfrak{n}^{i_q}_2$ are equivalent and irreducible $\Ad(K_1)$-modules. Thus $\Lambda|_{\mathfrak{m}_{i_q}}$ has the form
   \begin{equation}\label{equ88}
A|_{\mathfrak{m}_{i_q}}=\left (
\begin{array}{cc}
\mu^{i_q}_1\mathrm{Id} |_{\mathfrak{n}^{i_q}_1}& A^{i_q}_{21}  \\
A^{i_q}_{12} & \mu^{i_q}_2\mathrm{Id}|_{\mathfrak{n}^{i_q}_2}
\end{array}
\right ),  \mu^{i_q}_1, \mu^{i_q}_2> 0,
\end{equation}
where $A^{i_q}_{jk}:\mathfrak{n}^{i_q}_j\rightarrow \mathfrak{n}^{i_q}_k, \; j\neq k \in \{1, 2\}$ is an $\Ad(K_1)$-equivalent map, and
 $\Lambda X=\Lambda|_{\mathfrak{m}_{i_q}}X=\mu^{i_q}_j X+A^{i_q}_{jk} X$ for any $X \in\mathfrak{n}^{i_q}_j \subset\mathfrak{m}_{i_q}, j \in \{1, 2\}$.
 Since $\mathfrak{m}_{i_{q+1}}$ is irreducible as $\Ad(K_1)$-module, it follows that $\Lambda|_{\mathfrak{m}_{i_{q+1}}}=\mu_{i_{q+1}}\mathrm{Id}|_{\mathfrak{m}_{i_{q+1}}}$ for some $\mu_{i_{q+1}}> 0$.

Since  $G/K_1$ is a g.o. space, by Corollary \ref{C3} it follows that for any $X\in \mathfrak{n}$  there exists a $k \in \mathfrak{k}_1$  such that $[k+X, \Lambda X] \in \mathfrak{k}_1$.
 We choose non zero vectors $X_1\in  \mathfrak{n}^{i_q}_1 \subset \mathfrak{m}_{i_{q}}$ and  $X_2 \in \mathfrak{m}_{i_{q+1}}$, with $[k+X_1+X_2, \Lambda(X_1+X_2)] \in \mathfrak{k}_1$.
 Since $\Lambda(X_1+X_2)=\Lambda|_{\mathfrak{m}_{i_q}}X_1+\Lambda|_{\mathfrak{m}_{i_{q+1}}}X_2=\mu^{i_q}_1 X_1+A^{i_q}_{12} X_1+\mu_{i_{q+1}} X_2$,  we obtain that

\begin{eqnarray*}
&&[k+X_1+X_2, \Lambda(X_1+X_2)]=[k+X_1+X_2, \mu^{i_q}_1 X_1+A^{i_q}_{12} X_1+\mu_{i_{q+1}} X_2]\\
&&=[k, \mu^{i_q}_1 X_1+A^{i_q}_{12} X_1]+ [k, \mu_{i_{q+1}} X_2]+(\mu_{i_{q+1}}-\mu^{i_q}_1)[X_1, X_2]+[X_1, A^{i_q}_{12} X_1]+[X_2, A^{i_q}_{12} X_1].
\end{eqnarray*}

Also,  since $\zeta _{q}+\zeta_{{q+1}}\neq \pm \zeta _{q}$,  $\zeta _{q}+\zeta_{{q+1}}\neq\pm \zeta_{{q+1}}$  and  $\zeta _{q}-\zeta_{{q+1}}\neq \pm \zeta _{q}$,  $\zeta _{q}-\zeta_{{q+1}}\neq\pm \zeta_{{q+1}}$,
   it follows that
   \begin{equation*}
   [k, \mu^{i_q}_1 X_1+A^{i_q}_{12} X_1]\in \mathfrak{m}_{i_q}, [k, \mu_{i_{q+1}} X_2]\in \mathfrak{m}_{i_{q+1}}, [X_1, A^{i_q}_{12} X_1]\in \mathfrak{k}\oplus \mathfrak{m}_{j_1}, (i_q\neq j_1 \neq i_{q+1}),
   \end{equation*}
    \begin{equation*}
    [X_2, A^{i_q}_{12} X_1] \in \mathfrak{m}_{j_2}\oplus \mathfrak{m}_{j_3},  (\mu_{i_{q+1}}-\mu^{i_q}_1)[X_1, X_2]\in \mathfrak{m}_{j_2}\oplus \mathfrak{m}_{j_3}, (i_q\neq j_2, j_3 \neq i_{q+1}).
    \end{equation*}
    Then it follows that
 \begin{equation*}
 ([k, \mu^{i_q}_1 X_1+A^{i_q}_{12} X_1]+[k, \mu_{i_{q+1}} X_2])\cap([X_1, A^{i_q}_{12} X_1]+[X_2, A^{i_q}_{12} X_1]+(\mu_{i_{q+1}}-\mu^{i_q}_1)[X_1, X_2])=\{ 0 \}.
 \end{equation*}

 Next, we prove that $A^{i_q}_{12}=0$ and $\mu^{i_q}_1=\mu_{i_{q+1}}$.

 By Remark \ref{R1} we have  that there exists $\sqrt{-1}h_{\Lambda_l}\in \mathfrak{s}$ such that $\mathfrak{n}^{i_q}_2=[\sqrt{-1}h_{\Lambda_l}, \mathfrak{n}^{i_q}_1]$ and $\mathfrak{n}^{i_q}_1=[\sqrt{-1}h_{\Lambda_l}, \mathfrak{n}^{i_q}_2]$. Assume that $A^{i_q}_{jk}\neq 0$, $j\neq k \in \{1, 2\}$ and  $A^{i_q}_{jk}$ is the matrix representation of the following map
 \begin{equation*}
 \ad(\sqrt{-1}h_{\Lambda_l}): \; \mathfrak{n}^{i_q}_j\rightarrow \mathfrak{n}^{i_q}_k, \quad j\neq k \in \{1, 2\}.
 \end{equation*}

 Since $\mathfrak{m}_{i_q}$ is reducible as $\Ad(K_1)$-module, by Lemma \ref{Le3} and Proposition \ref{PP4} we obtain that there exist $\alpha, \beta \in R^+_{i_q}$ such that $\alpha\mid_{\mathfrak{a}_1}=-\beta\mid_{\mathfrak{a}_1}$ and $\alpha(h)=\beta(h)$ for any $h\in \mathfrak{s}$.   Since $\zeta_q, \zeta_{q+1}$ are adjacent, it  follows that either   $\zeta_q +\zeta_{q+1} \in R_{\mathfrak{t}}$ or  $\zeta_{q+1} -\zeta_q \in R_{\mathfrak{t}}$. Assume that $\zeta_q +\zeta_{q+1} \in R_{\mathfrak{t}}$, this implies that there exists $ \gamma \in R^+_M$  such that $\kappa(\gamma)=\zeta_{q+1}$ and $\kappa(\alpha+\gamma)\neq 0$.
 Hence we have  that there exists $j \neq q, q+1$ such that $\kappa(\alpha+\gamma)=\xi_j \in R^+_{\mathfrak{t}}$.

We now distinguish two cases and we will get a contradiction to $A_{jk}^{i_q}\ne 0$.

{\bf (i)} Assume that $\alpha=\beta$.

 Since $\alpha=\beta$ and $\alpha\mid_{\mathfrak{a}_1}=-\beta\mid_{\mathfrak{a}_1}$, it follows that $\alpha(h)=\beta(h)=0$ for any $h\in \mathfrak{a}_1$, and this implies that $\dim \mathfrak{m}_{i_q}=2$. By Remark \ref{R2} we have $\mathfrak{n}^{i_q}_1=\mathbb{R}A_{\alpha}, \mathfrak{n}^{i_q}_2=\mathbb{R}B_{\alpha}$.

   We choose $X_1=A_{\alpha}$ and  $X_2=A_{\gamma}$. Then we have
  \begin{eqnarray*}
 &&[X_1, A^{i_q}_{12}X_1]+[X_2, A^{i_q}_{12} X_1]+(\mu^{i_q}_1-\mu_{i_{q+1}})[X_1, X_2]=[A_{\alpha}, [\sqrt{-1}h_{\Lambda_l}, A_{\alpha}]\\
 &&+
 [A_{\gamma}, [\sqrt{-1}h_{\Lambda_l}, A_{\alpha}]]+(\mu_{i_{q+1}}-\mu^{i_q}_1)[A_{\alpha}, A_{\gamma}].
 \end{eqnarray*}

  It is easy to get  that
   \begin{eqnarray*}
   &&[X_1, A^{i_q}_{12}X_1]=2\sqrt{-1}\alpha(h_{\Lambda_l})h_{\alpha} \in \mathfrak{s},\\
   \\
   &&[X_2, A^{i_q}_{12} X_1]=\alpha(h_{\Lambda_l})(N_{\gamma, \alpha}B_{\gamma+\alpha}+N_{\gamma, -\alpha}B_{\gamma-\alpha}),\\
   \\
    &&(\mu_{i_{q+1}}-\mu^{i_q}_1)[X_1, X_2]=(\mu_{i_{q+1}}-\mu^{i_q}_1)(N_{\gamma, \alpha}A_{\gamma+\alpha}+N_{\gamma,-\alpha}A_{\gamma-\alpha}).
    \end{eqnarray*}
    Therefore we have that
   \begin{equation}\label{equ123}
    ([X_1, A^{i_q}_{12}X_1]+[X_2, A^{i_q}_{12} X_1])\cap (\mu^{i_q}_1-\mu_{i_{q+1}})[X_1, X_2]=\{0\}.
    \end{equation}
Since $([X_1, A^{i_q}_{12}X_1]+[X_2, A^{i_q}_{12} X_1])\notin \mathfrak{k}_1$ and $ (\mu^{i_q}_1-\mu_{i_{q+1}})[X_1, X_2]\notin \mathfrak{k}_1$, by  (\ref{equ123}) it follows that  $(\mu_{i_{q+1}}-\mu^{i_q}_1)[X_1, X_2]=2(\mu_{i_{q+1}}-\mu^{i_q}_1)(N_{\gamma, \alpha}A_{\gamma+\alpha}+N_{\gamma,-\alpha}A_{\gamma-\alpha})=0$ and $[X_1, A^{i_q}_{12}X_1]=0$. It is easy to check that $(\mu_{i_{q+1}}-\mu^{i_q}_1)[X_1, X_2]=2(\mu_{i_{q+1}}-\mu^{i_q}_1)(N_{\gamma, \alpha}A_{\gamma+\alpha}+N_{\gamma,-\alpha}A_{\gamma-\alpha})=0$ if and only if $\mu_{i_{q+1}}=\mu^{i_q}_1$. It is obvious that $[X_1, A^{i_q}_{12}X_1]=2\sqrt{-1}\alpha(h_{\Lambda_l})h_{\alpha}\neq0$, which is a contradiction, so we have $A^{i_q}_{12}=A^{i_q}_{21}=0$.
\smallskip

    {\bf (ii)} Assume that  $\alpha\neq \beta$.

     We will prove that $ (\gamma+\beta)\neq(\gamma\pm\alpha)\neq (\gamma-\beta)$.

 Assuming that $(\gamma+\beta)=(\gamma+\alpha)$, it follows that $(\gamma+\beta)(h)=(\gamma+\alpha)(h)$ for any $h\in \mathfrak{a}_1$. This implies that $ \beta(h)=\alpha(h)$ for any $h\in \mathfrak{a}_1$. Since  $-\beta(h)=\alpha(h)$ for any $h\in \mathfrak{a}_1$, it follows  that $\beta(h)=\alpha(h)=0$ for any $h\in\mathfrak{a}_1$, this implies  that $\dim\mathfrak{m}_{i_q}=2$. By Remark \ref{R3} and Proposition \ref{PP3} we have $\beta=-\alpha$. But $\beta(h)=\alpha(h)$ for any $h\in\mathfrak{s}$, this implies that $\alpha, \beta \in R^{+}_{i_q}$,  which is a contradiction. Hence we have $(\gamma+\beta)\neq(\gamma+\alpha)$.

 Assuming that $(\gamma+\beta)=(\gamma-\alpha)$, it follows that $(\gamma+\beta)(h)=(\gamma-\alpha)(h)$ for any $h\in \mathfrak{s}$. This implies that $ \beta(h)=-\alpha(h)$ for any $h\in \mathfrak{s}$.  But  $\beta(h)=\alpha(h)$ for any $h\in \mathfrak{s}$, which is a contradiction. Hence we have $(\gamma+\beta)\neq(\gamma-\alpha)$.

 We prove that $(\gamma\pm\alpha)\neq (\gamma-\beta)$ by the same method as above.

We choose $X_1=A_{\alpha}+A_{-\beta}$ and $X_2=A_{\gamma}$.



 Since
 \begin{eqnarray*}
    &&(\mu_{i_{q+1}}-\mu^{i_q}_1)[X_1, X_2]=(\mu_{i_{q+1}}-\mu^{i_q}_1)(N_{\gamma, \alpha}A_{\gamma+\alpha}+N_{\gamma,-\alpha}A_{\gamma-\alpha}+N_{\gamma, -\beta}A_{\gamma-\beta}+N_{\gamma, \beta}A_{\gamma+\beta}),\\
    \\
 &&[X_1, A^{i_q}_{12}X_1]=(2\sqrt{-1}\alpha(h_{\Lambda_l})(h_{\alpha}-h_{-\beta})-2\alpha(h_{\Lambda_l})N_{\alpha, -\beta}B_{\alpha-\beta})\in \mathfrak{k}, \\
 \\
 &&
 [X_2, A^{i_q}_{12} X_1]=\alpha(h_{\Lambda_l})(N_{\gamma, \alpha}B_{\gamma+\alpha}+N_{\gamma, -\alpha}B_{\gamma-\alpha}-N_{\gamma, -\beta}B_{\gamma-\beta}-N_{\gamma, \beta}B_{\gamma+\beta})\notin \mathfrak{k},
 \end{eqnarray*}

it follows that
\begin{equation}
(\mu_{i_{q+1}}-\mu^{i_q}_1)[X_1, X_2]\cap([X_1, A^{i_q}_{12}X_1]+[X_2, A^{i_q}_{12} X_1])=\{0\}.
\end{equation}

 Since  $ (\gamma+\beta)\neq(\gamma\pm\alpha)\neq (\gamma-\beta)$, it is easy to check that
    $(\mu_{i_{q+1}}-\mu^{i_q}_1)[X_1, X_2]=(\mu_{i_{q+1}}-\mu^{i_q}_1)(N_{\gamma, \alpha}A_{\gamma+\alpha}+N_{\gamma,-\alpha}A_{\gamma-\alpha}+N_{\gamma, -\beta}A_{\gamma-\beta}+N_{\gamma, \beta}A_{\gamma+\beta})\in \mathfrak{k}_1$ if and only if  $\mu^{i_q}_1=\mu_{i_{q+1}}$.

Since $[X_1, A^{i_q}_{12}X_1]\in \mathfrak{k}$ and $[X_2, A^{i_q}_{12} X_1]\notin\mathfrak{k}$, it follows that
\begin{equation}\label{d1}
[X_1, A^{i_q}_{12}X_1]\cap[X_2, A^{i_q}_{12} X_1]=\{0\}.
\end{equation}
This implies that $[X_2, A^{i_q}_{12} X_1]\in \mathfrak{k}_1$
if and only if
\begin{equation}\label{u1}
 [X_2, A^{i_q}_{12} X_1]=\alpha(h_{\Lambda_l})(N_{\gamma, \alpha}B_{\gamma+\alpha}+N_{\gamma, -\alpha}B_{\gamma-\alpha}-N_{\gamma, -\beta}B_{\gamma-\beta}-N_{\gamma, \beta}B_{\gamma+\beta})=0.
 \end{equation}
Since  $ (\gamma+\beta)\neq(\gamma\pm\alpha)\neq (\gamma-\beta)$ and $N_{\gamma, \alpha}\neq0$,
it follows that
 \begin{equation*}
[X_2, A^{i_q}_{12} X_1]=\alpha(h_{\Lambda_l})(N_{\gamma, \alpha}B_{\gamma+\alpha}+N_{\gamma, -\alpha}B_{\gamma-\alpha}-N_{\gamma, -\beta}B_{\gamma-\beta}-N_{\gamma, \beta}B_{\gamma+\beta})\neq 0,
 \end{equation*}
which contradicts with (\ref{u1}). Hence we get $A^{i_{q}}_{12}=0.$

  We  prove  $\mu^{i_q}_2=\mu_{i_{q+1}}$ by the same method. Consequently, we have $A^{i_q}_{jk}=0$, for $j\neq k \in \{1, 2\}$ and $\mu^{i_q}_1=\mu^{i_q}_2=\mu_{i_{q+1}}$.

\medskip

{\bf (b2)} Assume that $\mathfrak{m}_{i_{q+1}}$ is reducible as $\Ad(K_1)$-module.

 By Lemma \ref{L2} it follows that $\mathfrak{m}_{i_{q+1}}=\mathfrak{n}^{i_{q+1}}_1\oplus \mathfrak{n}^{i_{q+1}}_2$, where $\mathfrak{n}^{i_{q+1}}_1, \mathfrak{n}^{i_{q+1}}_2$ are equivalent and irreducible $\Ad(K_1)$-modules. Thus $\Lambda|_{\mathfrak{m}_{i_{q+1}}}$ has the form
   \begin{equation}\label{equ88}
A|_{\mathfrak{m}_{i_{q+1}}}=\left (
\begin{array}{cc}
\mu^{i_{q+1}}_1\mathrm{Id} |_{\mathfrak{n}^{i_{q+1}}_1}& A^{i_{q+1}}_{21}  \\
A^{i_{q+1}}_{12} & \mu^{i_{q+1}}_2\mathrm{Id}|_{\mathfrak{n}^{i_{q+1}}_2}
\end{array}
\right ),  \mu^{i_{q+1}}_1,\  \mu^{i_{q+1}}_2> 0,
\end{equation}
where $A^{i_{q+1}}_{jk}:\mathfrak{n}^{i_{q+1}}_j\rightarrow \mathfrak{n}^{i_{q+1}}_k, \; j\neq k \in \{1, 2\}$ is  an $\Ad(K_1)$-equivalent map. We can prove that $A^{i_{q+1}}_{21}=A^{i_{q+1}}_{12}=0$ and $\mu^{i_{q+1}}_1=\mu^{i_{q+1}}_2=\mu_{i_q}$ by the same method as above, where $\Lambda\mid_{\mathfrak{m}_{i_q}}=\mu_{i_q}\mathrm{Id}\mid_{\mathfrak{m}_{i_q}}$.

\medskip
  {\bf (c)} Both $\mathfrak{m}_{i_q}, \mathfrak{m}_{i_{q+1}}$ are reducible as $\Ad(K_1)$-modules.

  By Lemma \ref{L2} it follows that $\mathfrak{m}_{i_q}=\mathfrak{n}^{i_q}_1\oplus \mathfrak{n}^{i_q}_2$ and $\mathfrak{m}_{i_{q+1}}=\mathfrak{n}^{i_{q+1}}_1\oplus \mathfrak{n}^{i_{q+1}}_2$, where $\mathfrak{n}^{i_q}_1, \mathfrak{n}^{i_q}_2$ are equivalent and irreducible $\Ad(K_1)$-modules, and $\mathfrak{n}^{i_{q+1}}_1, \mathfrak{n}^{i_{q+1}}_2$ are  equivalent and irreducible $\Ad(K_1)$-modules. Thus $\Lambda|_{\mathfrak{m}_{i_q}}$ has  the form
   \begin{equation}\label{equ81}
A|_{\mathfrak{m}_{i_q}}=\left (
\begin{array}{cc}
\mu^{i_q}_1\mathrm{Id} |_{\mathfrak{n}^{i_q}_1}& A^{i_q}_{21}  \\
A^{i_q}_{12} & \mu^{i_q}_2\mathrm{Id}|_{\mathfrak{n}^{i_q}_2}
\end{array}
\right ), \  \mu^{i_q}_1, \ \mu^{i_q}_2> 0,
\end{equation}
and   $\Lambda|_{\mathfrak{m}_{i_{q+1}}}$ has the form
 \begin{equation}\label{equ82}
A|_{\mathfrak{m}_{i_{q+1}}}=\left (
\begin{array}{cc}
\mu^{i_{q+1}}_1\mathrm{Id} |_{\mathfrak{n}^{i_{q+}}_1}& A^{i_{q+1}}_{21}  \\
A^{i_{q+1}}_{12} & \mu^{i_{q+1}}_2\mathrm{Id}|_{\mathfrak{n}^{i_{q+1}}_2}
\end{array}
\right ), \  \mu^{i_{q+1}}_1,\  \mu^{i_{q+1}}_2> 0.
\end{equation}

Since  $G/K_1$ is a g.o. space, by Corollary \ref{C3} it follows that for any $X\in \mathfrak{n}$  there exists a $k \in \mathfrak{k}_1$  such that $[k+X, \Lambda X] \in \mathfrak{k}_1$.
 We choose non zero vectors $X_1\in  \mathfrak{n}^{i_q}_1 \subset \mathfrak{m}_{i_{q}}$ and  $X_2 \in \mathfrak{n}^{i_{q+1}}_1\subset \mathfrak{m}_{i_{q+1}}$, with $[k+X_1+X_2, \Lambda(X_1+X_2)] \in \mathfrak{k}_1$.
 Since $\Lambda(X_1+X_2)=\Lambda|_{\mathfrak{m}_{i_q}}X_1+\Lambda|_{\mathfrak{m}_{i_{q+1}}}X_2=\mu^{i_q}_1 X_1+A^{i_q}_{12} X_1+\mu^{i_{q+1}}_1 X_2+A^{i_{q+1}}_{12} X_2$,  we obtain that
\begin{eqnarray*}
&&[k+X_1+X_2, \Lambda(X_1+X_2)]=[k+X_1+X_2, \mu^{i_q}_1 X_1+A^{i_q}_{12} X_1+\mu^{i_{q+1}}_1 X_2+A^{i_{q+1}}_{12} X_2]\\
&&=[k, \mu^{i_q}_1 X_1+A^{i_q}_{12} X_1]+ [k, \mu^{i_{q+1}}_1 X_2+A^{i_{q+1}}_{12} X_2]+(\mu^{i_{q+1}}_1-\mu^{i_q}_1)[X_1, X_2]+[X_1, A^{i_q}_{12} X_1]\\
&&+[X_1, A^{i_{q+1}}_{12} X_2]+[X_2,  A^{i_{q+1}}_{12} X_2]+ [X_2, A^{i_q}_{12} X_1].
\end{eqnarray*}
Also,  since $\zeta _{q}+\zeta_{{q+1}}\neq \pm \zeta _{q}$,  $\zeta _{q}+\zeta_{{q+1}}\neq\pm \zeta_{{q+1}}$  and  $\zeta _{q}-\zeta_{{q+1}}\neq \pm \zeta _{q}$,  $\zeta _{q}-\zeta_{{q+1}}\neq\pm \zeta_{{q+1}}$,
   it follows that
   \begin{equation*}
   [k, \mu^{i_q}_1 X_1+A^{i_q}_{12} X_1]\in \mathfrak{m}_{i_q}, [k, \mu_{i_{q+1}} X_2]\in \mathfrak{m}_{i_{q+1}}, [X_1, A^{i_q}_{12} X_1]\in \mathfrak{k}\oplus \mathfrak{m}_{j_1}, (i_q\neq j_1 \neq i_{q+1}),
   \end{equation*}
   \begin{equation*}
   [X_2, A^{i_{q+1}}_{12} X_2]\in \mathfrak{k}\oplus \mathfrak{m}_{j_2}, (i_q\neq j_2 \neq i_{q+1}, j_2 \neq j_1),
   \end{equation*}
    \begin{equation*}
    ([X_2, A^{i_q}_{12} X_1]+ [X_1, A^{i_{q+1}}_{12} X_2] +  (\mu^{i_{q+1}}_1-\mu^{i_q}_1)[X_1, X_2])\in \mathfrak{m}_{j_3}\oplus \mathfrak{m}_{j_4}, (i_q\neq j_3, j_4 \neq i_{q+1}).
    \end{equation*}
    Then it follows that
 \begin{eqnarray*}
 &&([k, \mu^{i_q}_1 X_1+A^{i_q}_{12} X_1]+[k, \mu_{i_{q+1}} X_2])\cap((\mu^{i_{q+1}}_1-\mu^{i_q}_1)[X_1, X_2]+[X_1, A^{i_q}_{12} X_1]\\
 &&
+[X_1, A^{i_{q+1}}_{12} X_2]+[X_2,  A^{i_{q+1}}_{12} X_2]+ [X_2, A^{i_q}_{12} X_1])=\{ 0 \}.
 \end{eqnarray*}

Since $\mathfrak{m}_{i_q}$ is reducible as $\Ad(K_1)$-module, by Lemma \ref{Le3} and Proposition \ref{PP4} we have there exist $\alpha_1, \beta_1 \in R^+_{i_q}$ such that $\alpha_1|_{\mathfrak{a}_1}=-\beta_1|_{\mathfrak{a}_1}$ and $\alpha_1(h)=\beta_1(h)$ for any $h \in \mathfrak{s}$.  Since $\mathfrak{m}_{i_{q+1}}$ is reducible as $\Ad(K_1)$-module, by Lemma \ref{Le3} and Proposition \ref{PP4} we have that there exist $\alpha_2, \beta_2 \in R^+_{i_{q+1}}$ such that  $\alpha_2|_{\mathfrak{a}_1}=-\beta_2|_{\mathfrak{a}_1}$ and  $\alpha_2(h)=\beta_2(h)$ for any $h \in \mathfrak{s}$ and $\kappa(\alpha_1+\alpha_2)\neq 0$. It follows that $\alpha_1+\alpha_2 \in R^+_M$.

 Next, we will show that $A^{i_q}_{12}=A^{i_{q+1}}_{12}=0$ and $\mu^{i_q}_1=\mu^{i_{q+1}}_1$.

{\bf Case (c1).}  Assume that  $A^{i_q}_{12}=0$.

 We will prove that $A^{i_{q+1}}_{12}=0$ and $\mu^{i_q}_1=\mu^{i_{q+1}}_1$.

 Assume that $A^{i_{q+1}}_{jk}\neq 0$, $j\neq k \in \{1, 2\}$ and  $A^{i_{q+1}}_{jk}$ is the matrix representation of the following map
 \begin{equation*}
 \ad(\sqrt{-1}h_{\Lambda_l}): \; \mathfrak{n}^{i_{q+1}}_j\rightarrow \mathfrak{n}^{i_{q+1}}_k, \quad j\neq k \in \{1, 2\}.
 \end{equation*}

 {\bf (i)} Assume that $\alpha_1=\beta_1, \alpha_2\neq\beta_2$.

 We choose $X_1=A_{\alpha_1}$ and $X_2=A_{\alpha_2}+A_{-\beta_2}$. Then we have
\begin{eqnarray*}
 &&(\mu^{i_{q+1}}_1-\mu^{i_q}_1)[X_1, X_2]+[X_1, A^{i_q}_{12} X_1]
+[X_1, A^{i_{q+1}}_{12} X_2]+[X_2,  A^{i_{q+1}}_{12} X_2]+ [X_2, A^{i_q}_{12} X_1]\\
&&=(\mu^{i_{q+1}}_1-\mu^{i_q}_1)[X_1, X_2]
+[X_1, A^{i_{q+1}}_{12} X_2]+[X_2,  A^{i_{q+1}}_{12} X_2]\\
&&=(\mu^{i_{q+1}}_1-\mu^{i_q}_1)[A_{\alpha_1}, A_{\alpha_2}+A_{-\beta_2}]+[A_{\alpha_1}, [\sqrt{-1}h_{\Lambda_l}, A_{\alpha_2}+A_{-\beta_2}]]\\
&&+[A_{\alpha_2}+A_{-\beta_2}, [\sqrt{-1}h_{\Lambda_l}, A_{\alpha_2}+A_{-\beta_2}]]\\
\end{eqnarray*}
and
 \begin{eqnarray*}\label{eq1}
&&(\mu^{i_{q+1}}_1-\mu^{i_q}_1)[X_1, X_2]=(\mu^{i_{q+1}}_1-\mu^{i_q}_1)(N_{\alpha_1, \alpha_2}A_{\alpha_1+\alpha_2}+N_{-\alpha_1, \alpha_2}A_{\alpha_1-\alpha_2}+N_{\alpha_1, -\beta_2}A_{\alpha_1-\beta_2}\\
&&+N_{-\alpha_1, -\beta_2}A_{\alpha_1+\beta_2}),\\
\\
&&[X_1, A^{i_{q+1}}_{12} X_2]=\alpha_2(h_{\Lambda_l})(N_{\alpha_1, \alpha_2}B_{\alpha_1+\alpha_2}+N_{\alpha_1, -\alpha_2}B_{\alpha_1-\alpha_2}-N_{\alpha_1, -\beta_2}B_{\alpha_1-\beta_2}-N_{\alpha_1, \beta_2}B_{\alpha_1+\beta_2}),\\
\\
&& [X_2,  A^{i_{q+1}}_{12} X_2]=2\sqrt{-1}\alpha_2(h_{\Lambda_l})(h_{\alpha_2}-h_{-\beta_2})-2\alpha_2(h_{\Lambda_l})N_{\alpha_2, -\beta_2}B_{\alpha_2-\beta_2}.
 \end{eqnarray*}
It follows that
\begin{equation}
(\mu^{i_{q+1}}_1-\mu^{i_q}_1)[X_1, X_2]\cap
([X_1, A^{i_{q+1}}_{12} X_2]+[X_2,  A^{i_{q+1}}_{12} X_2])=\{0\}.
\end{equation}

Since $\alpha_2\neq\beta_2$, by the same method as in {\bf (b)} we prove that $(\alpha_1\pm\alpha_2)\neq(\alpha_1\pm\beta_2)$.
It follows that $(\mu^{i_{q+1}}_1-\mu^{i_q}_1)[X_1, X_2]\in \mathfrak{k}_1$ if and only if $\mu^{i_{q+1}}_1=\mu^{i_q}_1$.

Since $[X_2,  A^{i_{q+1}}_{12} X_2]=(2\sqrt{-1}\alpha_2(h_{\Lambda_l})(h_{\alpha_2}-h_{-\beta_2})-2\alpha_2(h_{\Lambda_l})N_{\alpha_2, -\beta_2}B_{\alpha_2-\beta_2})\in \mathfrak{k}$ and $[X_1, A^{i_{q+1}}_{12} X_2]\notin \mathfrak{k}$,  it follows that
$$
[X_2,  A^{i_{q+1}}_{12} X_2]\cap[X_1, A^{i_{q+1}}_{12} X_2]=\{0\},
$$
this implies that
$
[X_1, A^{i_{q+1}}_{12} X_2] \in \mathfrak{k}_1
$
 if and only if
 $$[X_1, A^{i_{q+1}}_{12} X_2]=0.
 $$
 But $(\alpha_1\pm\alpha_2)\neq(\alpha_1\pm\beta_2),$
 it follows that $[X_1, A^{i_{q+1}}_{12} X_2]\neq0$, which is a contradiction. Hence we have $A^{i_{q+1}}_{jk}=0$.

Therefore if $A^{i_{q}}_{jk}=0$, we conclude that $A^{i_{q+1}}_{jk}=0$ and $\mu^{i_{q+1}}_1=\mu^{i_q}_1$.    Also, we prove  $\mu^{i_{q+1}}_1=\mu^{i_q}_2$ and  $\mu^{i_{q+1}}_2=\mu^{i_q}_1=\mu^{i_q}_2$ by the same method.

{\bf (ii)} Assume that  $\alpha_1\neq \beta_1, \alpha_2=\beta_2$.

We choose $X_1=A_{\alpha_1}+A_{-\beta_1}$ and $X_2=A_{\alpha_2}$.
If $A^{i_{q}}_{jk}=0$, we conclude that $A^{i_{q+1}}_{jk}=0$ and $\mu^{i_{q+1}}_1=\mu^{i_q}_1$ by the same method as in {\bf 1)}.    Also, we prove  $\mu^{i_{q+1}}_1=\mu^{i_q}_2$ and  $\mu^{i_{q+1}}_2=\mu^{i_q}_1=\mu^{i_q}_2$ by the same method.

{\bf (iii)} Assume that  $\alpha_1\neq \beta_1, \alpha_2\neq\beta_2$.

We choose $X_1=A_{\alpha_1}+A_{-\beta_1}$ and $X_2=A_{\alpha_2}+A_{-\beta_2}$.  Then we have
\begin{eqnarray*}\label{eq1}
&&(\mu^{i_{q+1}}_1-\mu^{i_q}_1)[X_1, X_2]=(\mu^{i_{q+1}}_1-\mu^{i_q}_1)(N_{\alpha_1, \alpha_2}A_{\alpha_1+\alpha_2}+N_{-\alpha_1, \alpha_2}A_{\alpha_1-\alpha_2}+N_{\alpha_1, -\beta_2}A_{\alpha_1-\beta_2}\\
&&+N_{-\alpha_1, -\beta_2}A_{\alpha_1+\beta_2}
+N_{-\beta_1, \alpha_2}A_{-\beta_1+\alpha_2}+N_{\beta_1, \alpha_2}A_{-\beta_1-\alpha_2}+N_{-\beta_1, -\beta_2}A_{-\beta_1-\beta_2}+N_{\beta_1, -\beta_2}A_{-\beta_1+\beta_2}),\\
\\
&&[X_2, A^{i_{q+1}}_{12} X_2]=[A_{\alpha_2}+A_{-\beta_2}, [\sqrt{-1}h_{\Lambda_k}, A_{\alpha_2}+A_{-\beta_2}]]=2\sqrt{-1}\alpha_2(h_{\Lambda_k})(h_{\alpha_2}-h_{-\beta_2})\\
&&-2\alpha_2(h_{\Lambda_k})N_{\alpha_2, -\beta_2}B_{\alpha_2-\beta_2},\\
\\
&&[X_1, A^{i_{q+1}}_{12} X_2]=[A_{\alpha_1}+A_{-\beta_1}, [\sqrt{-1}h_{\Lambda_l}, A_{\alpha_2}+A_{-\beta_2}]]=\alpha_2(h_{\Lambda_l})(N_{\alpha_1, \alpha_2}B_{\alpha_1+\alpha_2}+N_{\alpha_1, -\alpha_2}B_{\alpha_1-\alpha_2}\\
&&-N_{\alpha_1, -\beta_2}B_{\alpha_1-\beta_2}-N_{\alpha_1, \beta_2}B_{\alpha_1+\beta_2}
+N_{-\beta_1, \alpha_2}B_{-\beta_1+\alpha_2}+N_{-\beta_1, -\alpha_2}B_{-\beta_1-\alpha_2}-N_{-\beta_1, -\beta_2}B_{-\beta_1-\beta_2}\\
&&-N_{-\beta_1, \beta_2}B_{-\beta_1+\beta_2}).
 \end{eqnarray*}
  It follows that
\begin{equation*}
(\mu^{i_{q+1}}_1-\mu^{i_q}_1)[X_1, X_2]\cap
([X_2,  A^{i_{q+1}}_{12} X_2]+[X_1, A^{i_{q+1}}_{12} X_2])=\{0\}.
\end{equation*}

 Since $(-\beta_1\pm\alpha_2)\neq(\alpha_1\pm\alpha_2)\neq(-\beta_1\pm\beta_2)$ and  $(\alpha_1\pm\alpha_2)\neq(\alpha_1\pm\beta_2)$,  then we have that $(\mu^{i_{q+1}}_1-\mu^{i_q}_1)[X_1, X_2] \in \mathfrak{k}_1$ if and only if $\mu^{i_{q+1}}_1=\mu^{i_q}_1$.

 Since $(-\beta_1\pm\alpha_2)\neq(\alpha_1\pm\alpha_2)\neq(-\beta_1\pm\beta_2),$ and  $(\alpha_1\pm\alpha_2)\neq(\alpha_1\pm\beta_2)$,  it follows that
 \begin{eqnarray*}\label{eq1}
[X_2, A^{i_{q+1}}_{12} X_2]\in \mathfrak{k},\quad \quad \quad
[X_1, A^{i_{q+1}}_{12} X_2]\in \mathfrak{m}
 \end{eqnarray*}
 This implies that
 \begin{eqnarray*}\label{eq1}
 [X_2,  A^{i_{q+1}}_{12} X_2]\cap[X_1, A^{i_{q+1}}_{12} X_2]=0.
 \end{eqnarray*}
Therefore we have
 $
 [X_1, A^{i_{q+1}}_{12} X_2]\in \mathfrak{k}_1
 $
 if and only if
 \begin{eqnarray*}\label{eq1}
 &&[X_1, A^{i_{q+1}}_{12} X_2]=0.
 \end{eqnarray*}
  It is easy to check that
    \begin{eqnarray*}\label{eq1}
 &&[X_1, A^{i_{q+1}}_{12} X_2]\neq0,
 \end{eqnarray*}
   which is a contradiction. Hence we have $A^{i_{q+1}}_{jk}=0$.

 We prove $\mu^{i_{q+1}}_1=\mu^{i_q}_2, \mu^{i_{q+1}}_2=\mu^{i_q}_1, \mu^{i_{q+1}}_2=\mu^{i_q}_2$  by the same method as above.

Hence if $A^{i_{q}}_{jk}=0$, we prove that $A^{i_{q+1}}_{jk}=0$ and $\mu^{i_{q+1}}_1=\mu^{i_q}_1, \mu^{i_{q+1}}_1=\mu^{i_q}_2, \mu^{i_{q+1}}_2=\mu^{i_q}_1, \mu^{i_{q+1}}_2=\mu^{i_q}_2$.

If $A^{i_{q+1}}_{jk}=0$, by the same method as above we prove that $A^{i_{q}}_{jk}=0$ and $\mu^{i_{q+1}}_1=\mu^{i_q}_1, \mu^{i_{q+1}}_1=\mu^{i_q}_2, \mu^{i_{q+1}}_2=\mu^{i_q}_1, \mu^{i_{q+1}}_2=\mu^{i_q}_2$.

{\bf Case (c2).} We assume that  $A^{i_{q}}_{jk}\neq0$  and $A^{i_{q+1}}_{jk}\neq0$. We will prove that we get a contradiction.

Since $A^{i_{q}}_{jk}\neq0$  and $A^{i_{q+1}}_{jk}\neq0$,
we assume that $A^{i_{q}}_{jk}$ and  $A^{i_{q+1}}_{jk}$ are  the matrix representations of the following maps
 \begin{equation*}
 \ad(\sqrt{-1}h_{\Lambda_l}): \; \mathfrak{n}^{i_{q}}_j\rightarrow \mathfrak{n}^{i_{q}}_k, \quad j\neq k \in \{1, 2\}
 \end{equation*}
 and
 \begin{equation*}
 \ad(\sqrt{-1}h_{\Lambda_k}): \; \mathfrak{n}^{i_{q+1}}_j\rightarrow \mathfrak{n}^{i_{q+1}}_k, \quad j\neq k \in \{1, 2\}
 \end{equation*}
 respectively.

{\bf (i)} Assume that $\alpha_1=\beta_1, \alpha_2\neq\beta_2.$

 We choose $X_1=A_{\alpha_1}$ and $X_2=A_{\alpha_2}+A_{-\beta_2}$. Then we have
\begin{eqnarray*}\label{eq1}
&&(\mu^{i_{q+1}}_1-\mu^{i_q}_1)[X_1, X_2]=(\mu^{i_{q+1}}_1-\mu^{i_q}_1)(N_{\alpha_1, \alpha_2}A_{\alpha_1+\alpha_2}+N_{-\alpha_1, \alpha_2}A_{\alpha_1-\alpha_2}+N_{\alpha_1, -\beta_2}A_{\alpha_1-\beta_2}\\
&&+N_{-\alpha_1, -\beta_2}A_{\alpha_1+\beta_2}),\\
\\
&&[X_1, A^{i_{q}}_{12} X_1]=[A_{\alpha_1}, [\sqrt{-1}h_{\Lambda_l}, A_{\alpha_1}]]=2\sqrt{-1}\alpha_1(h_{\Lambda_1})h_{\alpha_1},\\
\\
&&[X_2, A^{i_{q+1}}_{12} X_2]=[A_{\alpha_2}+A_{-\beta_2}, [\sqrt{-1}h_{\Lambda_k}, A_{\alpha_2}+A_{-\beta_2}]]=2\sqrt{-1}\alpha_2(h_{\Lambda_k})(h_{\alpha_2}-h_{-\beta_2})\\
&&-2\alpha_2(h_{\Lambda_k})N_{\alpha_2, -\beta_2}B_{\alpha_2-\beta_2},\\
\\
&&[X_1, A^{i_{q+1}}_{12} X_2]=[A_{\alpha_1}, [\sqrt{-1}h_{\Lambda_l}, A_{\alpha_2}+A_{-\beta_2}]]=\alpha_2(h_{\Lambda_l})(N_{\alpha_1, \alpha_2}B_{\alpha_1+\alpha_2}+N_{\alpha_1, -\alpha_2}B_{\alpha_1-\alpha_2}\\
&&-N_{\alpha_1, -\beta_2}B_{\alpha_1-\beta_2}-N_{\alpha_1, \beta_2}B_{\alpha_1+\beta_2}),\\
\\
&&[X_2, A^{i_{q}}_{12} X_1]=[A_{\alpha_2}+A_{-\beta_2}, [\sqrt{-1}h_{\Lambda_k}, A_{\alpha_1}]]=
\alpha_1(h_{\Lambda_k})(N_{\alpha_2, \alpha_1}B_{\alpha_2+\alpha_1}+N_{\alpha_2, -\alpha_1}B_{\alpha_2-\alpha_1}
\\
&&
+N_{-\beta_2, \alpha_1}B_{-\beta_2+\alpha_1}+N_{-\beta_2, -\alpha_1}B_{-\beta_2-\alpha_1}).
 \end{eqnarray*}
  It follows that
\begin{equation*}
(\mu^{i_{q+1}}_1-\mu^{i_q}_1)[X_1, X_2]\cap
([X_1, A^{i_{q}}_{12} X_1]+[X_2,  A^{i_{q+1}}_{12} X_2]+[X_1, A^{i_{q+1}}_{12} X_2]+[X_2, A^{i_{q}}_{12} X_1])=\{0\}.
\end{equation*}

Since $\alpha_2\neq\beta_2 $, by the same method as in {\bf1)} we prove that  $(\alpha_1\pm\alpha_2)\neq(\alpha_1\pm\beta_2)$,   then we have that $(\mu^{i_{q+1}}_1-\mu^{i_q}_1)[X_1, X_2] \in \mathfrak{k}_1$ if and only if $\mu^{i_{q+1}}_1=\mu^{i_q}_1$.

 Since $[X_1, A^{i_{q}}_{12} X_1]+[X_2, A^{i_{q}}_{12} X_2])\in \mathfrak{k}$ and  $(\alpha_1\pm\alpha_2)\neq(\alpha_1\pm\beta_2)$,  it follows that
 \begin{eqnarray*}\label{eq1}
 &&([X_1, A^{i_{q}}_{12} X_1]+[X_2,  A^{i_{q+1}}_{12} X_2])\cap([X_1, A^{i_{q+1}}_{12} X_2]+[X_2, A^{i_{q}}_{12} X_1])=\{0\}.
 \end{eqnarray*}
 Then we have
 \begin{eqnarray*}\label{eq1}
 &&([X_1, A^{i_{q+1}}_{12} X_2]+[X_2, A^{i_{q}}_{12} X_1])\in \mathfrak{k}_1
 \end{eqnarray*}
 if and only if
 \begin{eqnarray*}\label{eq1}
 &&[X_1, A^{i_{q+1}}_{12} X_2]+[X_2, A^{i_{q}}_{12} X_1]=0.
 \end{eqnarray*}
  It is easy to check that
    \begin{eqnarray*}\label{eq1}
 &&([X_1, A^{i_{q+1}}_{12} X_2]+[X_2, A^{i_{q}}_{12} X_1])=0
 \end{eqnarray*}
 if and only if $\alpha_1(h_{\Lambda_k})=\alpha_2(h_{\Lambda_l})=0$. But $\alpha_1(h_{\Lambda_k})\neq0$ and $\alpha_2(h_{\Lambda_l})\neq0$,
   which is a contradiction. Hence we have that at least one of $A^{i_{q}}_{jk}, A^{i_{q+1}}_{jk}$ is equal zero.  By
{\bf Case c1} we have that $A^{i_{q}}_{jk}=A^{i_{q+1}}_{jk}=0$  and
 $\mu^{i_{q+1}}_1=\mu^{i_q}_1$. Similarly, we prove that $\mu^{i_{q+1}}_1=\mu^{i_q}_2, \mu^{i_{q+1}}_2=\mu^{i_q}_1, \mu^{i_{q+1}}_2=\mu^{i_q}_2$.
 \medskip

{\bf (ii)} Assume that  $\alpha_1\neq\beta_1, \alpha_2=\beta_2.$

 We choose $X_1=A_{\alpha_1}+A_{-\beta_1}$ and $X_2=A_{\alpha_2}$.  By the same method as above we get a contradiction. Hence we have that at least one of $A^{i_{q}}_{jk}, A^{i_{q+1}}_{jk}$ is equal zero.  By
{\bf Case c1} we have that $A^{i_{q}}_{jk}=A^{i_{q+1}}_{jk}=0$  and
 $\mu^{i_{q+1}}_1=\mu^{i_q}_1, \mu^{i_{q+1}}_1=\mu^{i_q}_2, \mu^{i_{q+1}}_2=\mu^{i_q}_1, \mu^{i_{q+1}}_2=\mu^{i_q}_2$.

 \medskip
{\bf (iii)} if $\alpha_1\neq\beta_1, \alpha_2\neq\beta_2.$

 We choose $X_1=A_{\alpha_1}+A_{-\beta_1}$ and $X_2=A_{\alpha_2}+A_{-\beta_2}$. Then we have
\begin{eqnarray*}\label{eq1}
&&(\mu^{i_{q+1}}_1-\mu^{i_q}_1)[X_1, X_2]=(\mu^{i_{q+1}}_1-\mu^{i_q}_1)(N_{\alpha_1, \alpha_2}A_{\alpha_1+\alpha_2}+N_{-\alpha_1, \alpha_2}A_{\alpha_1-\alpha_2}+N_{\alpha_1, -\beta_2}A_{\alpha_1-\beta_2}\\
&&+N_{-\alpha_1, -\beta_2}A_{\alpha_1+\beta_2}
+N_{-\beta_1, \alpha_2}A_{-\beta_1+\alpha_2}+N_{\beta_1, \alpha_2}A_{-\beta_1-\alpha_2}+N_{-\beta_1, -\beta_2}A_{-\beta_1-\beta_2}+N_{\beta_1, -\beta_2}A_{-\beta_1+\beta_2}),\\
\\
&&[X_1, A^{i_{q}}_{12} X_1]=[A_{\alpha_1}+A_{-\beta_1}, [\sqrt{-1}h_{\Lambda_l}, A_{\alpha_1}+A_{-\beta_1}]]=2\sqrt{-1}\alpha_1(h_{\Lambda_1})(h_{\alpha_1}-h_{-\beta_1})\\
&&-2\alpha_1(h_{\Lambda_l})N_{\alpha_1, -\beta_1}B_{\alpha_1-\beta_1},\\
\\
&&[X_2, A^{i_{q+1}}_{12} X_2]=[A_{\alpha_2}+A_{-\beta_2}, [\sqrt{-1}h_{\Lambda_k}, A_{\alpha_2}+A_{-\beta_2}]]=2\sqrt{-1}\alpha_2(h_{\Lambda_k})(h_{\alpha_2}-h_{-\beta_2})\\
&&-2\alpha_2(h_{\Lambda_k})N_{\alpha_2, -\beta_2}B_{\alpha_2-\beta_2},\\
\\
&&[X_1, A^{i_{q+1}}_{12} X_2]=[A_{\alpha_1}+A_{-\beta_1}, [\sqrt{-1}h_{\Lambda_l}, A_{\alpha_2}+A_{-\beta_2}]]=\alpha_2(h_{\Lambda_l})(N_{\alpha_1, \alpha_2}B_{\alpha_1+\alpha_2}+N_{\alpha_1, -\alpha_2}B_{\alpha_1-\alpha_2}\\
&&-N_{\alpha_1, -\beta_2}B_{\alpha_1-\beta_2}-N_{\alpha_1, \beta_2}B_{\alpha_1+\beta_2}
+N_{-\beta_1, \alpha_2}B_{-\beta_1+\alpha_2}+N_{-\beta_1, -\alpha_2}B_{-\beta_1-\alpha_2}-N_{-\beta_1, -\beta_2}B_{-\beta_1-\beta_2}\\
&&-N_{-\beta_1, \beta_2}B_{-\beta_1+\beta_2}),\\
\\
&&[X_2, A^{i_{q}}_{12} X_1]=[A_{\alpha_2}+A_{-\beta_2}, [\sqrt{-1}h_{\Lambda_k}, A_{\alpha_1}+A_{-\beta_1}]]=
\alpha_1(h_{\Lambda_k})(N_{\alpha_2, \alpha_1}B_{\alpha_2+\alpha_1}+N_{\alpha_2, -\alpha_1}B_{\alpha_2-\alpha_1}
\\
&&-N_{\alpha_2, -\beta_1}B_{\alpha_2-\beta_1}-N_{\alpha_2, \beta_1}B_{\alpha_2+\beta_1}
+N_{-\beta_2, \alpha_1}B_{-\beta_2+\alpha_1}+N_{-\beta_2, -\alpha_1}B_{-\beta_2-\alpha_1}-N_{-\beta_2, -\beta_1}B_{-\beta_2-\beta_1}\\
&&-N_{-\beta_2, \beta_1}B_{-\beta_2+\beta_1}).
 \end{eqnarray*}
  It follows that
\begin{equation*}
(\mu^{i_{q+1}}_1-\mu^{i_q}_1)[X_1, X_2]\cap
([X_1, A^{i_{q}}_{12} X_1]+[X_2,  A^{i_{q+1}}_{12} X_2]+[X_1, A^{i_{q+1}}_{12} X_2]+[X_2, A^{i_{q}}_{12} X_1])=\{0\}.
\end{equation*}

 Since $\alpha_1\neq\beta_1$ and $\alpha_2\neq \beta_2$, it follows that  $(-\beta_1\pm\alpha_2)\neq(\alpha_1\pm\alpha_2)\neq(-\beta_1\pm\beta_2)$ and $(\alpha_1\pm\alpha_2)\neq(\alpha_1\pm\beta_2)$,  then we have that $(\mu^{i_{q+1}}_1-\mu^{i_q}_1)[X_1, X_2] \in \mathfrak{k}_1$ if and only if $\mu^{i_{q+1}}_1=\mu^{i_q}_1$.

 Since $[X_1, A^{i_{q}}_{12} X_1]+[X_2, A^{i_{q}}_{12} X_2])\in \mathfrak{k}$ and  $(-\beta_1\pm\alpha_2)\neq(\alpha_1\pm\alpha_2)\neq(-\beta_1\pm\beta_2), $ $(\alpha_1\pm\alpha_2)\neq(\alpha_1\pm\beta_2)$,  it follows that
 \begin{eqnarray*}\label{eq1}
 &&([X_1, A^{i_{q}}_{12} X_1]+[X_2,  A^{i_{q+1}}_{12} X_2])\cap([X_1, A^{i_{q+1}}_{12} X_2]+[X_2, A^{i_{q}}_{12} X_1])=\{0\}.
 \end{eqnarray*}
 This implies that
 \begin{eqnarray*}\label{eq1}
 &&([X_1, A^{i_{q+1}}_{12} X_2]+[X_2, A^{i_{q}}_{12} X_1])\in \mathfrak{k}_1
 \end{eqnarray*}
 if and only if
 \begin{eqnarray*}\label{eq1}
 &&[X_1, A^{i_{q+1}}_{12} X_2]+[X_2, A^{i_{q}}_{12} X_1]=0.
 \end{eqnarray*}
  It is easy to check that
    \begin{eqnarray*}\label{eq1}
 &&([X_1, A^{i_{q+1}}_{12} X_2]+[X_2, A^{i_{q}}_{12} X_1])=0
 \end{eqnarray*}
 if and only if $\alpha_1(h_{\Lambda_k})=\alpha_2(h_{\Lambda_l})=0$. But $\alpha_1(h_{\Lambda_k})\neq0$ and $\alpha_2(h_{\Lambda_l})\neq0$,
   which is a contradiction. Hence we have that at least one of $A^{i_{q}}_{jk}, A^{i_{q+1}}_{jk}$ is equal zero.  By
{\bf Case c1} we have that $A^{i_{q}}_{jk}=A^{i_{q+1}}_{jk}=0$  and
 $\mu^{i_{q+1}}_1=\mu^{i_q}_1, \mu^{i_{q+1}}_1=\mu^{i_q}_2, \mu^{i_{q+1}}_2=\mu^{i_q}_1, \mu^{i_{q+1}}_2=\mu^{i_q}_2$.
 \medskip

 Therefore,  if   $G$-invariant metrics on $G/K_1$ defined  by (\ref{equ7}) are g.o. metrics, then we get that $A|_{\mathfrak{m}}$ is  a diagonal matrix and all diagonal elements are equal.
 \hfill $\Box$

\medskip
\medskip
\noindent
{\it Proof of  Corollary \ref{CC1}}.

If $\dim \mathfrak{s}=1$ we assume that $\Pi_M=\{\alpha_k\}$ and it follows that $\mathfrak{s}=\mathbb{R}\sqrt{-1}h_{\Lambda_k}$. Since there exists $j\in \{1, \dots, s\}$ such that $\mathfrak{m}_j$ is reducible as an $\Ad(K_1)$-module, we set $\mathfrak{m}_j=\mathfrak{n}^j_1+\mathfrak{n}^j_2$, where $\mathfrak{n}^j_1$ and $\mathfrak{n}^j_2$ are equivalent and irreducible as  $\Ad(K_1)$-modules. Therefore,  $G$-invariant metrics on $G/K_1$ are defined  by (\ref{equ7}). Then
Theorem \ref{T1} implies that a g.o. metric $g$ for a $M$-space $G/K_1$ is
\begin{equation}
\langle \cdot,\cdot \rangle=\mu\mathrm{Id}|_{\mathbb{R}h_{\Lambda_{k}}}+ \lambda B(\cdot,\cdot)|_{{\mathfrak{m}}}.
\end{equation}

Let  $V\in \mathfrak{s}$ and $X\in \mathfrak{m}$ be eigenvectors of the associate  operator $\Lambda$  with different eigenvalues $\mu, \lambda$, Proposition \ref{P3} implies that there exists
  $k\in \mathfrak{k}_1$  such that $[V, X]=\frac{\mu}{\mu-\lambda}[k, V]+\frac{\lambda}{\mu-\lambda}[k, X]$. Since  $\mathfrak{m}_j=\mathfrak{n}^j_1\oplus \mathfrak{n}^j_{2}$,  by Remark \ref{L1} we have  that  $\mathfrak{n}^j_2=[\sqrt{-1}h_{\Lambda_k}, \mathfrak{n}^j_1]$ and $\mathfrak{n}^j_1=[\sqrt{-1}h_{\Lambda_k}, \mathfrak{n}^j_2]$. We choose $V=\sqrt{-1}h_{\Lambda_k}$ and $X\in \mathfrak{n}^j_1$ and  assume that  $\mu\neq \lambda$. Then Proposition \ref{P3} implies that
  $$[\sqrt{-1}h_{\Lambda_k}, X]=\frac{\mu}{\mu-\lambda}[k, \sqrt{-1}h_{\Lambda_k}]+\frac{\lambda}{\mu-\lambda}[k, X].
  $$
   Since  $[k, \sqrt{-1}h_{\Lambda_k}]=0$, it follows that $[\sqrt{-1}h_{\Lambda_k}, X]=\frac{\lambda}{\mu-\lambda}[k, X]$. But  $[\sqrt{-1}h_{\Lambda_k}, X]\in \mathfrak{n}^j_2$ and $[k, X] \in \mathfrak{n}^j_1$.  This is a contradiction. Hence we get $\lambda=\mu$, and the conclusion follows.
\hfill $\Box$

\section{Proof of Theorem 2 and Corollary 2}

\noindent
 {\it Proof of part 1) of Theorem \ref{T2}}.

   Let  $\mathfrak{n}$  be  the tangent space $T_o(G/K_1)$ at $o=eK_1$.  Since  $\mathfrak{m}_1$ and $\mathfrak{m}_2$ are irreducible as $\Ad(K_1)$-modules,
 then we have that
  \begin{equation}\label{28}
 \mathfrak{n}=\mathfrak{s}\oplus \mathfrak{m}_1\oplus\mathfrak{m}_2
 \end{equation}
  is an $\Ad(K_1)$-irreducible decomposition.

 Let  $\langle \cdot,\cdot \rangle=B(\Lambda\cdot,\cdot)$ be an $\Ad({K}_1)$-invariant scalar product on $\mathfrak{n}$, where  $\Lambda$ is the associated operator. Then we have
 \begin{equation}\label{29}
\langle \cdot,\cdot \rangle=\mu\mathrm{Id}|_{\mathfrak{s}}+ \mu_1B(\cdot,\cdot)|_{{\mathfrak{m}}_1}+\mu_2 B(\cdot,\cdot)|_{{\mathfrak{m}}_2}, \ (\mu,  \mu_1, \mu_2 >0).
 \end{equation}

 By  Corollary \ref{C3}  $(G/K_1, g)$ is a g.o. space if and only if for every $V \in \mathfrak{s}$ and  $X\in \mathfrak{m}$ there exists $k\in \mathfrak{k}_1$ such that
 \begin{equation}\label{30}
[k+V+X, \Lambda(V+X)]\in \mathfrak{k}_1.
 \end{equation}
Let  $X=X_1+X_2, X_i\in \mathfrak{m}_i \ (i=1,2)$.
Then  (\ref{30}) is equivalent to
 \begin{equation}\label{31}
[k+V+X_1+X_2, \mu V+ \mu_1X_1+\mu_2X_2]\in \mathfrak{k}_1.
 \end{equation}

 Since $[k+V, V]=0$, (\ref{31}) reduces to
 \begin{equation}\label{32}
([k+V,  \mu_1X_1] + [k+V,  \mu_2X_2]+[X_1, \mu V]+[X_2, \mu V] +(\mu_1-\mu_2)[X_2, X_1])\in \mathfrak{k}_1.
 \end{equation}

 Since $[\mathfrak{m}_2, \mathfrak{m}_1]\subseteq \mathfrak{m}_1$, it follows that $([k+V,  \mu_1X_1] +[X_1, \mu V]+(\mu_1-\mu_2)[X_2, X_1]) \in \mathfrak{m}_1$ and $( [k+V,  \mu_2X_2]+[X_2, \mu V] )\in \mathfrak{m}_2$,   (\ref{32}) is equivalent to
  \begin{equation}\label{33}
[\mu_1 k+(\mu_1-\mu)V+(\mu_1-\mu_2)X_2, X_1] =0
 \end{equation}
 and
 \begin{equation}\label{34}
([\mu_2 k+(\mu_2-\mu)V, X_2])=0,
 \end{equation}
and this completes the proof.\hfill $\Box$
 \medskip
 \medskip

\noindent
{\it  Proof of part 2) of Theorem \ref{T2}}.

 We assume that $[\mathfrak{m}_1, \mathfrak{m}_1]\subseteq \mathfrak{k}\oplus \mathfrak{m}_2$ in the decomposition (\ref{2}), so it follows that  $\dim \mathfrak{m}_1 \neq 2$.

\smallskip
{\bf Case 1.} Assume that $\mathfrak{m}_1$ is reducible as $\Ad(K_1)$-module and $\mathfrak{m}_2$ is irreducible as $\Ad(K_1)$-module.

 Lemma  \ref{L2}  implies that
 $\mathfrak{m}_1=\mathfrak{n}^1_1\oplus \mathfrak{n}^1_2$, where $\mathfrak{n}^1_1$ and $\mathfrak{n}^1_2$ are equivalent and irreducible as $\Ad(K_1)$-modules. Then we have that  $\mathfrak{n} \cong T_o(G/K_1)=\mathfrak{s}\oplus \mathfrak{n}^1_1\oplus \mathfrak{n}^1_2 \oplus \mathfrak{m}_2$ is the   $\Ad(K_1)$-irreducible  decomposition.

  Then $G$-invariant metrics on $G/K_1$ which are $\mathrm{Ad}({K}_1)$-invariant are defined by

 \begin{equation}\label{equ7f}
A=\left (
\begin{array}{cccc}
A|_{\mathfrak{s}}& 0 &  0 \\
0 & A|_{\mathfrak{m}_1} & 0  \\
 0 & 0 &  A|_{\mathfrak{m}_2}
\end{array}
\right )
\end{equation}

Since $\mathfrak{s}$ and $\mathfrak{m}_2$ are irreducible as $\Ad(K_1)$-modules, it follows that $\Lambda|_{\mathfrak{s}}=\mu\mathrm{Id}|_{\mathfrak{s}}$ and $\Lambda|_{\mathfrak{m}_2}=\lambda_2\mathrm{Id}|_{\mathfrak{m}_2}$, where $\mu, \lambda_2> 0$. $\Lambda|_{\mathfrak{m}_{1}}$ has  the form
   \begin{equation}\label{equ811}
A|_{\mathfrak{m}_{1}}=\left (
\begin{array}{cc}
\lambda^{1}_1\mathrm{Id} |_{\mathfrak{n}^{1}_1}& A_{21}  \\
A_{12} & \lambda^{1}_2\mathrm{Id}|_{\mathfrak{n}^{1}_2}
\end{array}
\right ),  \  \lambda^{1}_1, \ \lambda^{1}_2> 0,
\end{equation}
where $A_{jk}:\mathfrak{n}^{1}_j\rightarrow \mathfrak{n}^{1}_k, \; j\neq k \in \{1, 2\}$ is  $\Ad(K_1)$-equivalent map, and $\Lambda X=\Lambda|_{\mathfrak{m}_{1}}X=\lambda^{1}_j X+A_{jk} X$ for any $X \in {n}^{1}_j \subset\mathfrak{m}_{1}, j \in \{1, 2\}$.

 We need to show that $A_{jk}=0$, $j\neq k \in \{1, 2\}$,  so assume
 the contrary to get a contradiction.
 Let $\Pi_{M}=\{ \alpha_p\},  p\in \{1, \dots, l\}$.
 Assume that $A_{jk}\neq 0$ is the matrix representation of the following map
 \begin{equation}\label{ltz1}
 \ad(\sqrt{-1}h_{\Lambda_p}):\mathfrak{n}^{1}_j\rightarrow \mathfrak{n}^{1}_k, \; j\neq k \in \{1, 2\}.
 \end{equation}

 Set  $R^+_{\mathfrak{t}}=\{\xi_1, \xi_2\}$. Since $\dim \mathfrak{m}_1\neq 2$,  this implies   there exist $\alpha_1, \alpha_2\in R^+_1$  such that $\alpha_1+\alpha_2 \in R^+_2$, where $R^+_i, (i=1, 2)$ are defined by (\ref{w2}).  Lemma \ref{Le3} and Proposition \ref{PP4} implies  that there exist $\beta_1, \beta_2 \in R^+_1$  such that $\alpha_1\mid_{\mathfrak{a}_1}=-\beta_1\mid_{\mathfrak{a}_1}, \; \alpha_2\mid_{\mathfrak{a}_1}=-\beta_2\mid_{\mathfrak{a}_1}$  and $\alpha_1(h)=\beta_1(h), \alpha_2(h)=\beta_2(h)$ for any $h\in \mathfrak{s}$. Since $\dim \mathfrak{m}_1\neq2$,  Proposition \ref{PP3} and Remark \ref{R3} implies that  $\alpha_1\neq\beta_1$ and $\alpha_2\neq\beta_2$.

   Since $G/K_1$ is a g.o.space, by Corollary \ref{C3} it follows that for any $X\in \mathfrak{n}$  there exists a $k \in \mathfrak{k}_1$  such that $[k+X, \Lambda X] \in \mathfrak{k}_1$.
 We choose non zero vectors $X_1, X_2 \in  \mathfrak{n}^{1}_1$,  with  $[k+X_1+X_2, \Lambda(X_1+X_2)] \in \mathfrak{k}_1$. Then we have
\begin{eqnarray*}
&&[k+X_1+X_2, \lambda^{1}_1X_1+ A_{12}X_1+\lambda^{1}_1X_2+ A_{21}X_2]=[k, \lambda^{1}_1X_1+\lambda^{1}_1X_2]
+\\
&&[k, A_{21}X_2+A_{12}X_1]
+[X_1, A_{12}X_1]
 +[X_1, A_{21}X_2]+[X_2, A_{12}X_1]+[X_2, A_{21}X_2].
\end{eqnarray*}

 Since $\frak{n}^{1}_1$ and $\frak{n}^{1}_2$ are $\Ad(K_1)$-invariant, it follows that $[k, \lambda^{1}_1X_1+\lambda^{1}_1X_2]\subseteq \mathfrak{n}^{1}_1\subset \mathfrak{m}_1$ and  $[k, A_{21}X_2+A_{12}X_1]\subseteq \mathfrak{n}^{1}_2\subset \mathfrak{m}_1$.
 We choose $X_1=A_{\alpha_1}+A_{-\beta_1}, X_2=A_{\alpha_2}+A_{-\beta_2}$.
 Then we have that
 \begin{eqnarray*}
 &&[X_1, A_{12}X_1]=[A_{\alpha_1}+A_{-\beta_1}, [\sqrt{-1}h_{\Lambda_p}, A_{\alpha_1}+A_{-\beta_1}]=2\sqrt{-1}\alpha_1(h_{\Lambda_p})(h_{\alpha_1}-h_{-\beta_1})\\
 &&-2\alpha_1(h_{\Lambda_p})N_{\alpha_1, -\beta_1}B_{\alpha_1-\beta_1},\\
 \\
 &&[X_2, A_{21}X_2]=[A_{\alpha_2}+A_{-\beta_2}, [\sqrt{-1}h_{\Lambda_p}, A_{\alpha_2}+A_{-\beta_2}]=2\sqrt{-1}\alpha_2(h_{\Lambda_p})(h_{\alpha_2}-h_{-\beta_2})\\
 &&-2\alpha_2(h_{\Lambda_p})N_{\alpha_2, -\beta_2}B_{\alpha_2-\beta_2},\\
 \\
 &&[X_1, A_{21} X_2]=[A_{\alpha_1}+A_{-\beta_1}, [\sqrt{-1}h_{\Lambda_p}, A_{\alpha_2}+A_{-\beta_2}]]=\alpha_2(h_{\Lambda_p})(N_{\alpha_1, \alpha_2}B_{\alpha_1+\alpha_2}+N_{\alpha_1, -\alpha_2}B_{\alpha_1-\alpha_2}\\
&&-N_{\alpha_1, -\beta_2}B_{\alpha_1-\beta_2}-N_{\alpha_1, \beta_2}B_{\alpha_1+\beta_2}
+N_{-\beta_1, \alpha_2}B_{-\beta_1+\alpha_2}+N_{-\beta_1, -\alpha_2}B_{-\beta_1-\alpha_2}-N_{-\beta_1, -\beta_2}B_{-\beta_1-\beta_2}\\
&&-N_{-\beta_1, \beta_2}B_{-\beta_1+\beta_2}),\\
\\
&&[X_2, A_{12} X_1]=[A_{\alpha_2}+A_{-\beta_2}, [\sqrt{-1}h_{\Lambda_p}, A_{\alpha_1}+A_{-\beta_1}]]=
\alpha_1(h_{\Lambda_p})(N_{\alpha_2, \alpha_1}B_{\alpha_2+\alpha_1}+N_{\alpha_2, -\alpha_1}B_{\alpha_2-\alpha_1}
\\
&&-N_{\alpha_2, -\beta_1}B_{\alpha_2-\beta_1}-N_{\alpha_2, \beta_1}B_{\alpha_2+\beta_1}
+N_{-\beta_2, \alpha_1}B_{-\beta_2+\alpha_1}+N_{-\beta_2, -\alpha_1}B_{-\beta_2-\alpha_1}-N_{-\beta_2, -\beta_1}B_{-\beta_2-\beta_1}\\
&&-N_{-\beta_2, \beta_1}B_{-\beta_2+\beta_1}).
 \end{eqnarray*}

 Since $([X_1, A_{12}X_1]+[X_2, A_{21}X_2])\in \mathfrak{k}$, $([X_1, A_{21} X_2]+[X_2, A_{12} X_1])\in \mathfrak{m}_2\oplus \mathfrak{k}$ and $([k, \lambda^{1}_1X_1+\lambda^{1}_1X_2]+[k, A_{21}X_2+A_{12}X_1])\subset \mathfrak{m}_1$, it follows that
   \begin{equation*}
  ([k, \lambda^{1}_1X_1+\lambda^{1}_1X_2]+[k, A_{21}X_2+A_{12}X_1])
  \cap ([X_1, A_{12}X_1]+[X_2, A_{21}X_2]+[X_1, A_{21} X_2]+[X_2, A_{12} X_1])=\{0\}.
  \end{equation*}

   Since $\alpha_1\neq\beta_1$ and $\alpha_2\neq\beta_2$ we have $(-\beta_1\pm\alpha_2)\neq(\alpha_1\pm\alpha_2)\neq(-\beta_1\pm\beta_2)$  and $(\alpha_1\pm\alpha_2)\neq(\alpha_1\pm\beta_2)$, then we have $([X_1, A_{21} X_2]+[X_2, A_{12} X_1])\in  \mathfrak{k}_1$ if and only if $\alpha_1(h_{\Lambda_p})=\alpha_2(h_{\Lambda_p})=0$. But $\alpha_1(h_{\Lambda_p})\neq 0$ and $\alpha_2(h_{\Lambda_p})\neq0$, which is a contradiction. Therefore we have $A_{jk}=0, (j, k \in \{1, 2\})$.

   Now we prove that $\lambda^1_1=\lambda^2_1$ in (\ref{equ811}).

   Assume that $\lambda^1_1\neq\lambda^2_1$ in (\ref{equ811}). Since $G/K_1$ is a g.o.space, by Proposition \ref{P3} we have that for any $X_1 \in \mathfrak{n}^1_1,  X_2 \in \mathfrak{n}^1_2$ there exists a $k\in \mathfrak{k}_1$ such that
   $$
   [X_1, X_2]=\frac{\lambda^1_1}{\lambda^1_1-\lambda^1_2}[k, X_1]+\frac{\lambda^1_2}{\lambda^1_1-\lambda^1_2}[k, X_2].
   $$

     Lemma \ref{Le3} and Proposition \ref{PP4} implies  that there exist $\alpha_1, \beta_1 \in R^+_1$  such that $\alpha_1\mid_{\mathfrak{a}_1}=-\beta_1\mid_{\mathfrak{a}_1} $ and $\alpha_1(h)=\beta_1(h)$ for any $h\in \mathfrak{s}$. We choose $X_1=A_{\alpha_1}+A_{-\beta_1}, X_2=B_{\alpha_1}-B_{-\beta_1}$, then we have
     $[X_1, X_2]=2\sqrt{-1}(h_{\alpha_1}-h_{-\beta_1})-2N_{\alpha_1, -\beta_1}B_{\alpha_1-\beta_1}$, it follows that $[X_1, X_2]\neq0$ and $[X_1, X_2]\in \mathfrak{k}$. Since $\mathfrak{n}^1_1$ and $\mathfrak{n}^1_2$ are irreducible as $\Ad(K_1)$-modules, it follows that $(\frac{\lambda^1_1}{\lambda^1_1-\lambda^1_2}[k, X_1]+\frac{\lambda^1_2}{\lambda^1_1-\lambda^1_2}[k, X_2])\in \mathfrak{m}_1$, this is a contradiction. Therefore we have $\lambda^1_1=\lambda^1_2$.

  Therefore    $\Ad({K}_1)$-invariant scalar product on $\mathfrak{n}$ are reduced to
 \begin{equation}\label{e1}
\langle \cdot,\cdot \rangle=\mu\mathrm{Id}|_{\mathfrak{s}}+ \mu_1B(\cdot,\cdot)|_{{\mathfrak{m}}_1}+\mu_2 B(\cdot,\cdot)|_{{\mathfrak{m}}_2}, \ (\mu,  \mu_1, \mu_2 >0).
 \end{equation}

 Since $(G/K_1, g)$ is a g.o. space, assume that $\mu_1\neq \mu$, Proposition \ref{P3} implies that   for any $V\in \mathfrak{s}$ and $X\in \mathfrak{n}^1_1$ there exists
  $k\in \mathfrak{k}_1$ such that $[V, X]=\frac{\mu}{\mu-\mu_1}[k, V]+\frac{\mu_1}{\mu-\mu_1}[k, X]$.   Since   $[k, V]=0$, it follows that   $[V, X]=\frac{\mu_1}{\mu-\mu_1}[k, X]$. Since $[V, X]\in \mathfrak{n}^1_2$ and $[k, X]\in \mathfrak{n}^1_1$, this is a contradiction.
Hence we get that  $\mu_1=\mu$.

  Hence,  if $(G/K_1, g)$ is a g.o. space,  then the corresponding $\Ad({K}_1)$-invariant scalar product (\ref{equ7f})  is reduced to
  \begin{equation}\label{36}
\langle \cdot,\cdot \rangle=\mu B(\cdot,\cdot)|_{\mathfrak{s}\oplus \mathfrak{m}_1}+ \mu_2B(\cdot,\cdot)|_{{\mathfrak{m}}_2}.
 \end{equation}

\smallskip
{\bf Case 2.} Assume that $\mathfrak{m}_2$ is reducible as $\Ad(K_1)$-module and $\mathfrak{m}_1$ is irreducible as $\Ad(K_1)$-module.

By Lemma  \ref{L2}  we obtain that
 $\mathfrak{m}_2=\mathfrak{n}^2_1\oplus \mathfrak{n}^2_2$, where $\mathfrak{n}^2_1$ and $\mathfrak{n}^2_2$ are equivalent and irreducible as $\Ad(K_1)$-modules. It follows that   $\mathfrak{n} \cong T_o(G/K_1)=\mathfrak{s}\oplus \mathfrak{m}_1\oplus \mathfrak{n}^2_1\oplus \mathfrak{n}^2_2$ is the   $\Ad(K_1)$-irreducible  decomposition.

  Then $G$-invariant metrics on $G/K_1$ which are $\mathrm{Ad}({K}_1)$-invariant are defined by

 \begin{equation}\label{equ7e}
A=\left (
\begin{array}{cccc}
A|_{\mathfrak{s}}& 0 &  0 \\
0 & A|_{\mathfrak{m}_1} & 0  \\
 0 & 0 &  A|_{\mathfrak{m}_2}
\end{array}
\right )
\end{equation}

Since $\mathfrak{s}$ and $\mathfrak{m}_1$ are irreducible as $\Ad(K_1)$-modules, it follows that $\Lambda|_{\mathfrak{s}}=\mu\mathrm{Id}|_{\mathfrak{s}}$ and $\Lambda|_{\mathfrak{m}_1}=\lambda_1\mathrm{Id}|_{\mathfrak{m}_1}$, where $\mu, \lambda_1> 0$. $\Lambda|_{\mathfrak{m}_{2}}$ has  the form
   \begin{equation}\label{equ81}
A|_{\mathfrak{m}_{2}}=\left (
\begin{array}{cc}
\lambda^{2}_1\mathrm{Id} |_{\mathfrak{n}^{2}_1}& A_{21}  \\
A_{12} & \lambda^{2}_2\mathrm{Id}|_{\mathfrak{n}^{2}_2}
\end{array}
\right ),  \lambda^{2}_1, \lambda^{2}_2> 0,
\end{equation}
where $A_{jk}:\mathfrak{n}^{2}_j\rightarrow \mathfrak{n}^{2}_k, \; j\neq k \in \{1, 2\}$ is  $\Ad(K_1)$-equivalent map, and $\Lambda X=\Lambda|_{\mathfrak{m}_{2}}X=\lambda^{2}_j X+A_{jk} X$ for any $X \in \mathfrak{n}^{2}_j \subset\mathfrak{m}_{2}, j \in \{1, 2\}$.

 Assume on the contrary that $A_{jk}\neq0, \; j\neq k \in \{1, 2\}$ and  $\Pi_{M}=\{ \alpha_p\}$,
 $p\in \{1, \dots, l\}$.  Let  $A_{jk}$ be the matrix representation of the  map
 $$\ad(\sqrt{-1}h_{\Lambda_p}):\mathfrak{n}^{2}_j\rightarrow \mathfrak{n}^{2}_k, \; j\neq k \in \{1, 2\}.$$

We consider two cases:

\medskip
\noindent

 {\bf (a) } Let $\dim \mathfrak{m}_2=2$.

 By Remark \ref{R2} we obtain that there exists $\alpha \in R^+_M$ such that $\mathfrak{n}^2_1=\mathbb{R}A_{\alpha}$ and  $\mathfrak{n}^2_2=\mathbb{R}B_{\alpha}$. Since $\dim \mathfrak{m}_2=2$ it follows that $\beta\pm \alpha \notin R_M$ for any $\beta \in R_K$.

Since $G/K_1$ is a g.o.space, by Corollary \ref{C3} it follows that for any $X\in \mathfrak{n}$  there exists a $k \in \mathfrak{k}_1$  such that $[k+X, \Lambda X] \in \mathfrak{k}_1$.
 We choose nonzero vectors $X=Y+X_1$,  where  $Y=\sqrt{-1}h_{\Lambda_p}\in  \mathfrak{s}$ and  $X_1=A_{\alpha} \in \mathfrak{n}^{2}_1$.
   It is  $[k+Y+X_1, \Lambda(X_1+Y)] \in \mathfrak{k}_1$.
   Then we have that
 \begin{eqnarray*}
 &&[k+X_1+Y, \lambda^{1}_1X_1+ A_{12}X_1+\mu Y]=[k, \lambda^{1}_1X_1+\ad(\sqrt{-1}h_{\Lambda_p})(X_1)]\\
 &&\ \ +(\lambda^{2}_1-\mu)[X_1, Y]+[X_1, \ad(\sqrt{-1}h_{\Lambda_p})(X_1)]+[Y, \ad(\sqrt{-1}h_{\Lambda_p})(X_1)].
  \end{eqnarray*}
 It follows that
 $$
 [k, \lambda^{1}_1X_1+\ad(\sqrt{-1}h_{\Lambda_p})(X_1)]+(\lambda^{2}_1-\mu)[X_1, Y]+[Y, \ad(\sqrt{-1}h_{\Lambda_p})(X_1)]\subset \mathfrak{m}_2.
 $$
Since $\mathfrak{s}=\{h\in \mathfrak{a}\mid B(h, \Pi_K)=0\}$ and  $\beta\pm \alpha \notin R_M$ for any $\beta \in R_K$,  we have
that  $h_{\alpha} \in \mathfrak{s}$.
Since $\alpha({h_{\Lambda_p}})\neq 0$ and $[X_1, \ad(\sqrt{-1}h_{\Lambda_p})(X_1)]=[A_{\alpha}, \alpha({h_{\Lambda_p}})B_{\alpha}]=2\sqrt{-1}\alpha({h_{\Lambda_p}})h_{\alpha}$, it follows that $[X_1, \ad(\sqrt{-1}h_{\Lambda_p})(X_1)] \in \mathfrak{s}$, which implies that

 \begin{eqnarray*}
([k, \lambda^{1}_1X_1+A_{12}X_1]+(\lambda^{2}_1-\mu)[X_1, Y]
 +[Y, A_{12}X_1])\cap [X_1, A_{12}X_1]= \{0\},
 \end{eqnarray*}
 hence we have $[X_1, A_{12}X_1]=0$.
 But
 $$
 [X_1, A_{12}X_1]=2\sqrt{-1}\alpha({h_{\Lambda_p}})h_{\alpha}\neq 0,
 $$
 which is a contradiction, so we  have $A_{jk}=0$, $(j\neq k \in \{1, 2\})$. Thus $A|_{\mathfrak{m}_2}$ is a diagonal matrix.

 Therefore  $\Ad({K}_1)$-invariant g.o. metrics on $\mathfrak{n}$ are  reduced to
 \begin{equation}\label{e1}
\langle \cdot,\cdot \rangle=\mu\mathrm{Id}|_{\mathfrak{s}}+ \lambda_1B(\cdot,\cdot)|_{{\mathfrak{m}}_1}+\lambda^2_1 B(\cdot,\cdot)|_{{\mathfrak{n}}^2_1}+\lambda^2_2 B(\cdot,\cdot)|_{{\mathfrak{n}}^2_2}, \ (\mu,  \lambda_1, \lambda^2_1,  \lambda^2_1>0).
 \end{equation}

 Since $(G/K_1, g)$ is a g.o. space, we assume that $\lambda^2_1\neq \mu$. Proposition \ref{P3} implies that   for any $V\in \mathfrak{s}$ and $X\in \mathfrak{n}^2_1$ there exists
  $k\in \mathfrak{k}_1$ such that $[V, X]=\frac{\mu}{\mu-\lambda^2_1}[k, V]+\frac{\lambda^2_1}{\mu-\lambda^2_1}[k, X]$.   Since   $[k, V]=0$, it follows that   $[V, X]=\frac{\lambda^2_1}{\mu-\lambda^2_1}[k, X]$. Since $[V, X]\in \mathfrak{n}^2_2$ and $[k, X]\in \mathfrak{n}^2_1$, this is a contradiction.
Hence we get that  $\lambda^2_1=\mu$. By the same method we obtain that $\lambda^2_2=\mu$.

Hence,  if $(G/K_1, g)$ is a g.o. space,  then the corresponding $\Ad({K}_1)$-invariant scalar product (\ref{equ7f})  is reduced to
  \begin{equation}\label{36}
\langle \cdot,\cdot \rangle=\mu B(\cdot,\cdot)|_{\mathfrak{s}\oplus \mathfrak{m}_2}+ \mu_1B(\cdot,\cdot)|_{{\mathfrak{m}}_1}.
 \end{equation}
 \medskip
 {\bf (b)} Assume that $\dim \mathfrak{m}_2\neq2$.

 Since $\dim \mathfrak{m}_2\neq 2$,  this implies   there exist $\alpha_1, \alpha_2\in R^+_2$  such that $\alpha_1-\alpha_2 \in R^+_1$.  Lemma \ref{Le3} and Proposition \ref{PP4} implies  that there exist $\beta_1, \beta_2 \in R^+_2$  such that $\alpha_1\mid_{\mathfrak{a}_1}=-\beta_1\mid_{\mathfrak{a}_1}, \; \alpha_2\mid_{\mathfrak{a}_1}=-\beta_2\mid_{\mathfrak{a}_1}$  and $\alpha_1(h)=\beta_1(h), \alpha_2(h)=\beta_2(h)$ for any $h\in \mathfrak{s}$. Proposition \ref{PP3} and Remark \ref{R3} implies that $\alpha_1\neq\beta_1$ and $\alpha_2\neq\beta_2$.

   Since $G/K_1$ is a g.o.space, by Corollary \ref{C3} it follows that for any $X\in \mathfrak{n}$  there exists a $k \in \mathfrak{k}_1$  such that $[k+X, \Lambda X] \in \mathfrak{k}_1$.
 We choose non zero vectors $X_1, X_2 \in  \mathfrak{n}^{2}_1$,  with  $[k+X_1+X_2, \Lambda(X_1+X_2)] \in \mathfrak{k}_1$. Then we have

\begin{eqnarray*}
&&[k+X_1+X_2, \lambda^{2}_1X_1+ A_{12}X_1+\lambda^{2}_1X_2+ A_{12}X_2]=[k, \lambda^{2}_1X_1+\lambda^{2}_1X_2]
+\\
&&[k, A_{12}X_2+A_{12}X_1]
+[X_1, A_{12}X_1]
 +[X_1, A_{12}X_2]+[X_2, A_{12}X_1]+[X_2, A_{12}X_2].
\end{eqnarray*}

Since $\frak{n}^{2}_1$ and $\frak{n}^{2}_2$ are $\Ad(K_1)$-invariant, it follows that $[k, \lambda^{2}_1X_1+\lambda^{2}_1X_2]\subseteq \mathfrak{n}^{2}_1\subset \mathfrak{m}_2$ and  $[k, A_{21}X_2+A_{12}X_1]\subseteq \mathfrak{n}^{2}_2\subset \mathfrak{m}_2$.

 We choose $X_1=A_{\alpha_1}+A_{-\beta_1}, X_2=A_{\alpha_2}+A_{-\beta_2}$.
 Then we have that
 \begin{eqnarray*}
 &&[X_1, A_{12}X_1]=[A_{\alpha_1}+A_{-\beta_1}, [\sqrt{-1}h_{\Lambda_p}, A_{\alpha_1}+A_{-\beta_1}]=2\sqrt{-1}\alpha_1(h_{\Lambda_p})(h_{\alpha_1}-h_{-\beta_1})\\
 &&-2\alpha_1(h_{\Lambda_p})N_{\alpha_1, -\beta_1}B_{\alpha_1-\beta_1},\\
 \\
 &&[X_2, A_{12}X_2]=[A_{\alpha_2}+A_{-\beta_2}, [\sqrt{-1}h_{\Lambda_p}, A_{\alpha_2}+A_{-\beta_2}]=2\sqrt{-1}\alpha_2(h_{\Lambda_p})(h_{\alpha_2}-h_{-\beta_2})\\
 &&-2\alpha_2(h_{\Lambda_p})N_{\alpha_2, -\beta_2}B_{\alpha_2-\beta_2},\\
 \\
 &&[X_1, A_{12} X_2]=[A_{\alpha_1}+A_{-\beta_1}, [\sqrt{-1}h_{\Lambda_p}, A_{\alpha_2}+A_{-\beta_2}]]=\alpha_2(h_{\Lambda_p})(N_{\alpha_1, -\alpha_2}B_{\alpha_1-\alpha_2}\\
&&-N_{\alpha_1, -\beta_2}B_{\alpha_1-\beta_2}
+N_{-\beta_1, \alpha_2}B_{-\beta_1+\alpha_2}-N_{-\beta_1, \beta_2}B_{-\beta_1+\beta_2}),\\
\\
&&[X_2, A_{12} X_1]=[A_{\alpha_2}+A_{-\beta_2}, [\sqrt{-1}h_{\Lambda_p}, A_{\alpha_1}+A_{-\beta_1}]]=
\alpha_1(h_{\Lambda_p})(N_{\alpha_2, -\alpha_1}B_{\alpha_2-\alpha_1}
\\
&&-N_{\alpha_2, -\beta_1}B_{\alpha_2-\beta_1}
+N_{-\beta_2, \alpha_1}B_{-\beta_2+\alpha_1}
-N_{-\beta_2, \beta_1}B_{-\beta_2+\beta_1}).
 \end{eqnarray*}

 Since $([X_1, A_{12}X_1]+[X_2, A_{12}X_2])\in \mathfrak{k}$, $([X_1, A_{12} X_2]+[X_2, A_{12} X_1])\in \mathfrak{m}_1\oplus \mathfrak{k}$ and $([k, \lambda^{1}_1X_1+\lambda^{1}_1X_2]+[k, A_{21}X_2+A_{12}X_1])\subset \mathfrak{m}_2$, it follows that
   \begin{equation*}
  ([k, \lambda^{1}_1X_1+\lambda^{1}_1X_2]+[k, A_{21}X_2+A_{12}X_1])
  \cap ([X_1, A_{12}X_1]+[X_2, A_{21}X_2]+[X_1, A_{21} X_2]+[X_2, A_{12} X_1])=\{0\}.
  \end{equation*}

    Since $\alpha_1\neq\beta_1$ and $\alpha_2\neq\beta_2$ we have $(-\beta_1\pm\alpha_2)\neq(\alpha_1\pm\alpha_2)\neq(-\beta_1\pm\beta_2)$  and $(\alpha_1\pm\alpha_2)\neq(\alpha_1\pm\beta_2)$, then we have $([X_1, A_{21} X_2]+[X_2, A_{12} X_1])\in  \mathfrak{k}_1$ if and only if $\alpha_1(h_{\Lambda_p})=\alpha_2(h_{\Lambda_p})=0$. But $\alpha_1(h_{\Lambda_p})\neq 0$ and $\alpha_2(h_{\Lambda_p})\neq0$, which is a contradiction. Therefore we have $A_{jk}=0, (j, k \in \{1, 2\})$.

   Now we prove that $\lambda^2_1=\lambda^2_1$ in (\ref{equ811}).

   Assume that $\lambda^2_1\neq\lambda^2_1$ in (\ref{equ811}). Since $G/K_1$ is a g.o.space, by Proposition \ref{P3} we have that for any $X_1 \in \mathfrak{n}^2_1,  X_2 \in \mathfrak{n}^2_2$ there exists a $k\in \mathfrak{k}_1$ such that
   $$
   [X_1, X_2]=\frac{\lambda^2_1}{\lambda^2_1-\lambda^2_2}[k, X_1]+\frac{\lambda^2_2}{\lambda^2_1-\lambda^2_2}[k, X_2].
   $$

     Lemma \ref{Le3} and Proposition \ref{PP4} implies  that there exist $\alpha_1\neq\beta_1 \in R^+_2$  such that $\alpha_1\mid_{\mathfrak{a}_1}=-\beta_1\mid_{\mathfrak{a}_1} $ and $\alpha_1(h)=\beta_1(h)$ for any $h\in \mathfrak{s}$. We choose $X_1=A_{\alpha_1}+A_{-\beta_1}, X_2=B_{\alpha_1}-B_{-\beta_1}$, then we have
     $[X_1, X_2]=2\sqrt{-1}(h_{\alpha_1}-h_{-\beta_1})-2N_{\alpha_1, -\beta_1}B_{\alpha_1-\beta_1}$, it follows that $[X_1, X_2]\neq0$ and $[X_1, X_2]\in \mathfrak{k}$. Since $\mathfrak{n}^2_1$ and $\mathfrak{n}^2_1$ are irreducible as $\Ad(K_1)$-modules, it follows that $(\frac{\lambda^2_1}{\lambda^2_1-\lambda^2_2}[k, X_1]+\frac{\lambda^2_2}{\lambda^2_1-\lambda^2_2}[k, X_2])\in \mathfrak{m}_2$, this is a contradiction. Therefore we have $\lambda^2_1=\lambda^2_2$.
  Then   $\Ad({K}_1)$-invariant scalar product on $\mathfrak{n}$ are reduced to
 \begin{equation}\label{e1}
\langle \cdot,\cdot \rangle=\mu\mathrm{Id}|_{\mathfrak{s}}+ \mu_1B(\cdot,\cdot)|_{{\mathfrak{m}}_1}+\mu_2 B(\cdot,\cdot)|_{{\mathfrak{m}}_2}, \ (\mu,  \mu_1, \mu_2 >0).
 \end{equation}

 Since $(G/K_1, g)$ is a g.o. space, assume that $\mu_2\neq \mu$, Proposition \ref{P3} implies that   for any $V\in \mathfrak{s}$ and $X\in \mathfrak{n}^2_1$ there exists
  $k\in \mathfrak{k}_1$ such that $[V, X]=\frac{\mu}{\mu-\mu_2}[k, V]+\frac{\mu_2}{\mu-\mu_2}[k, X]$.   Since   $[k, V]=0$, it follows that   $[V, X]=\frac{\mu_2}{\mu-\mu_2}[k, X]$. Since $[V, X]\in \mathfrak{n}^2_2$ and $[k, X]\in \mathfrak{n}^2_1$, this is a contradiction.
Hence we get that  $\mu_2=\mu$.

 Hence,   if $(G/K_1, g)$ is a g.o. space,  then the corresponding $\Ad({K}_1)$-invariant scalar product (\ref{equ7e}) is   reduced to
  \begin{equation}\label{36}
\langle \cdot,\cdot \rangle=\mu B(\cdot,\cdot)|_{\mathfrak{s}\oplus \mathfrak{m}_2}+ \lambda_1B(\cdot,\cdot)|_{{\mathfrak{m}}_1}.
 \end{equation}
\hfill $\Box$

\noindent
{\it  Proof of part 3) of Theorem \ref{T2}}.

Set $\mathfrak{m}_i=\mathfrak{n}^i_1\oplus\mathfrak{n}^i_2$, $i=1, 2$. It follows that
$\Lambda|_{\mathfrak{m}_{i}}$ has  the form
   \begin{equation}\label{equ81}
A|_{\mathfrak{m}_{i}}=\left (
\begin{array}{cc}
\lambda^{i}_1\mathrm{Id} |_{\mathfrak{n}^{i}_1}& A^i_{21}  \\
A^i_{12} & \lambda^{i}_2\mathrm{Id}|_{\mathfrak{n}^{i}_2}
\end{array}
\right ),  \lambda^{i}_1, \lambda^{i}_2> 0,
\end{equation}
where $A^i_{jk}:\mathfrak{n}^{i}_j\rightarrow \mathfrak{n}^{i}_k, \; j\neq k \in \{1, 2\}$ is  an $\Ad(K_1)$-equivalent map.

If $(G/K_1, g)$ is a g.o. space,  we can prove that
$A|_{\mathfrak{m}_1}$  is a diagonal matrix using the same method as the proof of  {\bf Case 1} of part 2) of Theorem \ref{T2}   and $\lambda^1_1=\lambda^1_2=\mu$,
we can also  prove that
$A|_{\mathfrak{m}_2}$  is  diagonal matrixe  using the same method as  the proof of  {\bf Case 2} of part 2) of Theorem \ref{T2}  and  $\lambda^2_1=\lambda^2_2=\mu$,  and this completes the proof.
\hfill $\Box$

\medskip
\medskip

 \medskip
 \noindent
 {\it Proof of Corollary \ref{C2}}.

Since $\dim\mathfrak{m}_2=2$, it follows that $\mathfrak{m}_2$ is reducible as $\Ad({K}_1)$-module by Remark \ref{R2}. If $(G/K_1, g)$ is a g.o. space,   case 2) of Theorem \ref{T2} implies that $\Ad({K}_1)$-invariant scalar product on $\mathfrak{n}$ is
  \begin{equation}\label{36}
\langle \cdot,\cdot \rangle=\mu B(\cdot,\cdot)|_{\mathfrak{s}\oplus \mathfrak{m}_2}+ \mu_1B(\cdot,\cdot)|_{{\mathfrak{m}}_1}, \; (\mu, \mu_1 > 0).
 \end{equation}

 By Corollary \ref{C3}  $(G/K_1, g)$ is a g.o. space if and only if for any $V\in \mathfrak{s}$  and for any $X_1\in \mathfrak{m}_1, X_2\in \mathfrak{m}_2$, there exists $k \in \mathfrak{k}_1$ such that
 \begin{equation}\label{38}
 [k+V+X_1+X_2, \mu (V+X_2)+\mu_1X_1]  \in \mathfrak{k}_1.
 \end{equation}

 Since $\dim\mathfrak{m}_2=2$,  it follows that $[k, X_2]=0$. Also, $[k+V, V]=0$, so (\ref{38})  reduces to
 \begin{equation}\label{39}
 [\mu_1k+(\mu_1-\mu)V+(\mu_1-\mu)X_2, X_1]  \in \mathfrak{k}_1.
 \end{equation}

 Since $\mu_1 \neq \mu$, (\ref{39})  is equivalent to
  \begin{equation}\label{40}
 [\frac{\mu_1}{\mu_1-\mu}k+V+X_2, X_1]  \in \mathfrak{k}_1.
 \end{equation}
Since $[\mathfrak{m}_2, \mathfrak{m}_1]\subseteq \mathfrak{m}_1$, it follows that $[\frac{\mu_1}{\mu_1-\mu}k+V+X_2, X_1]  \in \mathfrak{m}_1$,  then we obtain  that
 \begin{equation}\label{41}
 [\frac{\mu_1}{\mu_1-\mu}k+V+X_2, X_1]=0.
 \end{equation}\qed

\section{Proof of Theorem 3}

 \medskip
 \noindent
 {\it Proof of part 2) of Theorem \ref{T3}}.

Since  $s=1$ in the decomposition (\ref{L2}), it follows that $\dim \mathfrak{s}=1$. We assume that $\Pi_M=\{\alpha_p\}$ and it follows that $\mathfrak{s}=\mathbb{R}\sqrt{-1}h_{\Lambda_p}$. Since $\mathfrak{m}$ is reducible as $\Ad(K_1)$-module, by Lemma \ref{L2} we have  $\mathfrak{m}=\mathfrak{n}_1+\mathfrak{n}_2$, where $\mathfrak{n}_1$ and $\mathfrak{n}_2$ are equivalent and irreducible as  $\Ad(K_1)$-modules. Therefore,  $G$-invariant metrics on $G/K_1$ are defined  by \begin{equation}\label{equ77}
A=\left (
\begin{array}{ccccc}
A|_{\mathfrak{s}}& 0 \\
0 & A|_{\mathfrak{m}}
\end{array}
\right ),
\end{equation}
where $A|_{\mathfrak{s}}=\Lambda\mid_{\mathfrak{s}}=\lambda\mathrm{Id}\mid_{\mathfrak{s}}, (\lambda >0)$,  and $A|_{\mathfrak{m}}$ has the form

 \begin{equation}\label{equ88}
A|_{\mathfrak{m}}=\left (
\begin{array}{cc}
\mu_1\mathrm{Id} |_{\mathfrak{n}_1}& A_{21}  \\
A_{12} & \mu_2\mathrm{Id}|_{\mathfrak{n}_2}
\end{array}
\right ),  \mu_1, \mu_2> 0.
\end{equation}
 The  block matrices $A_{12}$ and  $A_{21}$ correspond to  $\Ad(K_1)$-equivariant  maps $\phi_1: \mathfrak{n}_1\rightarrow \mathfrak{n}_2$ and $\phi_2: \mathfrak{n}_2\rightarrow \mathfrak{n}_1$ respectively.
 Moreover, the symmetry of $\Lambda$ implies that $A_{12}=A_{21}$. Consequently, for any vector $X_j\in \mathfrak{n}_j\subset \mathfrak{m}, j=1, 2$, it is

 \begin{equation}\label{equ8}
 \Lambda X_j=\Lambda|_{\mathfrak{m}}X_j=\mu_jX_j+A_{jk}X_j,\quad A_{jk}:\mathfrak{n}_j\rightarrow \mathfrak{n}_k, \;j\neq k \in \{1, 2\}.
 \end{equation}

 Lemma \ref{Le3} implies that there  exist $\alpha, \beta\in R^+_M$ such that $\alpha|_{\mathfrak{a}_1}=-\beta|_{\mathfrak{a}_1}$ and $\alpha(h)=\beta(h)$ for any $h \in \mathfrak{s}$, where $\alpha, \beta $ are the lowest and highest roots respectively.

 Now we prove the necessity of part 2) of Theorem \ref{T3}.

   Assume that  $A_{jk}\neq0, \; j\neq k \in \{1, 2\}$.  Let  $A_{jk}$ be the matrix representation of the  map
 $$\ad(\sqrt{-1}h_{\Lambda_p}):\mathfrak{n}_j\rightarrow \mathfrak{n}_k, \; j\neq k \in \{1, 2\}.$$

 {\bf1)} Assume that  $\alpha=\beta$.

 Since $G/K_1$ is a g.o. space, by Corollary \ref{C3} we have that for any $X=(X_1+Y)\in \mathfrak{n}$ there exists a  $k\in \mathfrak{k}_1$  such that $[k+X, \Lambda X]\in \mathfrak{k}_1$, where $X_1\in \mathfrak{n}_1$ and $Y\in \mathfrak{s}$. Since $\Lambda X=\lambda Y +\mu_1X_1+A_{12}X_1$, then we obtain that
 \begin{eqnarray*}
 &&[k+X, \Lambda X]=[k+X_1+Y, \lambda Y +\mu_1X_1+A_{12}X_1]=[k, \mu_1X_1+A_{12}X_1]+(\lambda-\mu_1)[X, Y]\\
 &&+[X_1, A_{12}X_1]+[Y, A_{12}X_1].
  \end{eqnarray*}

 Since $\alpha=\beta$, by Remark \ref{R3} we have $\dim \mathfrak{m}=2$, By Remark \ref{R2} we have $\mathfrak{n}_1=\mathbb{R}A_{\alpha_{p}}$ and $\mathfrak{n}_2=\mathbb{R}B_{\alpha_{p}}$. We choose $X_1=A_{\alpha_{p}}$ and $Y=\sqrt{-1}h_{\Lambda_{p}}$.
 It follows that
 \begin{eqnarray*}
 &&([k, \mu_1X_1+A_{12}X_1]+(\lambda-\mu_1)[X, Y]+[Y, A_{12}X_1])\in \mathfrak{m}, \quad\quad
 [X_1, A_{12}X_1]\in\mathfrak{s},
  \end{eqnarray*}
 which implies that
 \begin{eqnarray*}
([k, \lambda^{1}_1X_1+A_{12}X_1]+(\lambda-\mu_1)[X_1, Y]
 +[Y, A_{12}X_1])\cap [X_1, A_{12}X_1]= \{0\}.
 \end{eqnarray*}
 It follows that  $[X_1, A_{12}X_1]\in \mathfrak{k}_1$ if and only if $[X_1, A_{12}X_1]=0$.
 But
 $$
 [X_1, A_{12}X_1]=[A_{\alpha_p}, [\ad(\sqrt{-1}h_{\Lambda_p}), A_{\alpha_p}]]=2\sqrt{-1}\alpha({h_{\Lambda_p}})h_{\alpha_p}\neq 0,
 $$
 which is a contradiction, so we  have $A_{jk}=0$, $(j\neq k \in \{1, 2\})$. Thus $A|_{\mathfrak{m}}$ is a diagonal matrix.

 Therefore  $\Ad({K}_1)$-invariant g.o. metrics on $\mathfrak{n}$ are  reduced to
 \begin{equation}\label{e1}
\langle \cdot,\cdot \rangle=\lambda\mathrm{Id}|_{\mathfrak{s}}+\mu_1 B(\cdot,\cdot)|_{{\mathfrak{n}}_1}+\mu_2 B(\cdot,\cdot)|_{{\mathfrak{n}}_2}, \ (\lambda,   \mu_1,  \mu_2>0).
 \end{equation}

 Since $(G/K_1, g)$ is a g.o. space, we assume that $\mu_1\neq \lambda$. Proposition \ref{P3} implies that   for any $V\in \mathfrak{s}$ and $X\in \mathfrak{n}_1$ there exists
  $k\in \mathfrak{k}_1$ such that $[V, X]=\frac{\lambda}{\lambda-\mu_1}[k, V]+\frac{\mu_1}{\lambda-\mu_1}[k, X]$.   Since   $[k, V]=0$, it follows that   $[V, X]=\frac{\lambda}{\lambda-\mu_1}[k, X]$. Since $[V, X]\in \mathfrak{n}_2$ and $[k, X]\in \mathfrak{n}_1$, this is a contradiction.
Hence we get that  $\mu_1=\lambda$. By the same method we obtain that $\mu_2=\lambda$. Therefore we have that  g.o. metric on $G/K_1$ is the standard metric.

{\bf 2)} Assume that  $\alpha\neq\beta.$

Since  $\alpha\neq\beta$,  by Remark \ref{R3} it follows that  $\dim \mathfrak{m}\neq 2$,  this implies   there exist $\alpha_1, \alpha_2\in R^+_M$  such that $\alpha_1+\alpha_2 \in R^+_M$.  Lemma \ref{Le3} and Proposition \ref{PP4} implies  that there exist $\beta_1, \beta_2 \in R^+_M$  such that $\alpha_1\mid_{\mathfrak{a}_1}=-\beta_1\mid_{\mathfrak{a}_1}, \; \alpha_2\mid_{\mathfrak{a}_1}=-\beta_2\mid_{\mathfrak{a}_1}$  and $\alpha_1(h)=\beta_1(h), \alpha_2(h)=\beta_2(h)$ for any $h\in \mathfrak{s}$. Since $\dim \mathfrak{m}\neq2$,  Proposition \ref{PP3} and Remark \ref{R3} imply that  $\alpha_1\neq\beta_1$ and $\alpha_2\neq\beta_2$.

Since $G/K_1$ is a g.o. space, by Corollary \ref{C3} we have that for any $X=(X_1+X_2)\in \mathfrak{n}$ there exists a  $k\in \mathfrak{k}_1$  such that $[k+X, \Lambda X]\in \mathfrak{k}_1$, where $X_1, X_2\in \mathfrak{n}_1$.

We choose $X_1=A_{\alpha_1}+A_{-\beta_1}, X_2=A_{\alpha_2}+A_{-\beta_2}$. By the same method as the proof of Case 1 of  part 2) of Theorem \ref{T2} we prove that $A_{12}=A_{21}=0$ and $\mu_1=\mu_2=\lambda$. Hence we have that  g.o. metric on $M$-space $G/K_1$ is the standard metric.

The sufficiency  of part 2) of Theorem \ref{T3} is obvious.

\hfill $\Box$

\end{document}